\documentclass[11pt,a4paper,twoside]{article}
\usepackage[utf8]{inputenc}
\usepackage{amsmath,amsfonts,amssymb,amsthm,bbm,latexsym,mathrsfs}
\usepackage{graphicx,color,epsfig,fancyhdr,dsfont}
\usepackage{enumerate}
\usepackage{hyperref}
\usepackage{indentfirst}
\usepackage[all]{hypcap}
\usepackage{placeins}
\usepackage[affil-it]{authblk}
\usepackage{color}
\usepackage[includeheadfoot,margin=3.5cm]{geometry}
\usepackage{subcaption}

\allowdisplaybreaks

\newtheorem{theorem}{Theorem}[section]
\newtheorem{definition}[theorem]{Definition}
\newtheorem{lemma}[theorem]{Lemma}
\newtheorem{corollary}[theorem]{Corollary}
\newtheorem{proposition}[theorem]{Proposition}
\newtheorem{remark}[theorem]{Remark}
\newtheorem{example}[theorem]{Example}

\numberwithin{equation}{section}

\newcommand{\N}{{\mathbb{N}}}                     
\newcommand{\Z}{{\mathbb{Z}}}                     
\newcommand{\R}{{\mathbb{R}}}                     
\newcommand{\E}{{\mathbb{E}}}                     
\newcommand{{\X}}{{\mathbb{X}}}                     
\newcommand{\cL}{{\mathcal{L}}}                   
\newcommand{\apm}{{\bar{\tilde{\rho}}}}           

\newcommand{\dt}{{\Delta t}}
\newcommand{\dW}{{\Delta W}}

\newcommand{\norm}[1]{\left\lVert#1\right\rVert}
\newcommand{\abs}[1]{\left|#1\right|}
\newcommand{\rbrac}[1]{\left(#1\right)}
\newcommand{\sbrac}[1]{\left[#1\right]}

\newcommand{\hx}{\widehat{X}}
\newcommand{\hxit}[1]{\widehat{X}_{-k\tau+#1\Delta t}^{-k\tau}}

\title{Ergodic Numerical Approximation to Periodic Measures of Stochastic Differential Equations}

\author[1]{Chunrong Feng}
\author[1]{Yu Liu}
\author[1,2]{Huaizhong Zhao}
\affil[1]{Department of Mathematical Sciences, Durham
	University, DH1 3LE, UK}
\affil[2]{Research Centre for Mathematics and Interdisciplinary Sciences, Shandong University, Qingdao, Shandong, 266237, China}
\affil[ ]{chunrong.feng@durham.ac.uk, yu.liu3@durham.ac.uk, huaizhong.zhao@durham.ac.uk}
\date{}

\begin{document}
	
	\maketitle
	
\begin{abstract}
	In this paper, we consider numerical approximation to periodic measure of a time periodic stochastic differential equations (SDEs) under weakly dissipative condition. For this we first study the existence of the periodic measure $\rho_t$ and the large time behaviour of $\mathcal{U}(t+s,s,x):=\E\phi(X_{t}^{s,x})-\int\phi d\rho_t,$ where $X_t^{s,x}$ is the solution of the SDEs and $\phi$ is a test function being smooth and of polynomial growth at infinity. We prove $\mathcal{U}$ and all its spatial derivatives decay to 0 with exponential rate on time $t$ in the sense of average on initial time $s$. We also prove the existence and the geometric ergodicity of the periodic measure of the discretized semi-flow from the Euler-Maruyama scheme and moment estimate of any order when the time step is sufficiently small (uniform for all orders). We thereafter obtain that the weak error for the numerical scheme of infinite horizon is of the order $1$ in terms of the time step. We prove that the choice of step size can be uniform for all test functions $\phi$. Subsequently we are able to estimate the average periodic measure with ergodic numerical schemes.

	\medskip
	
\end{abstract}

\pagestyle{fancy}
\fancyhf{}
\fancyhead[LE,RO]{\thepage}
\fancyhead[LO]{\small{Ergodic Approximation to Periodic Measures of SDEs}}
\fancyhead[RE]{\small{C. Feng, Y. Liu and H. Zhao}}


\section{Introduction}
Random periodicity is ubiquitous in the real world from daily temperature process to economic cycles. The concepts of random periodic paths and periodic measures were introduced  and their ergodicity was obtained recently (\cite{Feng-Wu-Zhao2016},\cite{Feng-Zhao2012},\cite{Feng-Zhao2018},\cite{Feng-Zhao-Zhou2011},\cite{Zhao-Zheng2009}). They are two different indispensable ways in the pathwise sense and in distributions 
to describe random periodicity.  The ``equivalence'' of the random periodic solutions and periodic measures and their characterisation in terms of purely imaginary eigenvalues of the infinitesimal 
generator of the Markovian semi-group were obtained in \cite{Feng-Zhao2018}. The presence of pure imaginary eigenvalues distinguishes the random periodic processes/periodic measures regime from that of the stationary processes/mixing invariant measures, in the latter case the 
Koopman-von Neumann Theorem says the infinitesimal generator has a unique eigenvalue $0$ on the imaginary axis.

As in the case of deterministic dynamical systems where periodic motion has been in the central stage of its study, the relevance of random periodic paths and periodic measures to theoretical and applied problems arising in stochastic dynamical systems has begun to be realised. In particular, there has been progress in the study of some topics in stochastic dynamics e.g. 
bifurcations (Wang \cite{Wang2014}), random attractors (Bates, Lu and Wang \cite{Bates-Lu-Wang2014}), 
stochastic resonance (Cherubini, Lamb, Rasmussen and Sato \cite{Cherubini-Lamb-Rasmussen-Sato2017}, Feng, Zhao and Zhong \cite{Feng-Zhao-Zhong2019measure},\cite{Feng-Zhao-Zhong2019resonance}), random horseshoes (Huang, Lian and Lu \cite{Huang-Lian-Lu2019}), modelling the El N\^ino phenomenon (Chekroun, Simonnet and Ghil \cite{Chekroun-Simonnet-Ghil2011}), isochronicity of stochastic oscillations (Engel and Kuehn \cite{Engel-Kuehn2019}), and invariant measures of quasi-periodic stochastic systems (Feng, Qu and Zhao \cite{Feng-Qu-Zhao}).

However, it is difficult to construct random periodic solutions explicitly for many problems. So numerical approximation is critical in the study of stochastic dynamics in addition to the study of random periodic dynamics theory. There are numerous works on numerical analysis of SDEs on a finite horizon 
(\cite{Kloeden-Platen1991},\cite{Milstein1995},\cite{Jentzen-Kloeden2010},\cite{Milstein-Tretyakov2004}). 
A numerical analysis of approximation to invariant measures of SDEs through discretizing the pull-back,  
was given in \cite{mattingly-stuart-higham},\cite{Talay1990},\cite{Talay-Tubaro1990},\cite{Tocino-Ardanuy2002},\cite{Yevik-Zhao2011},\cite{Yuan-Mao2004}. 
Numerical approximations to stable zero solutions of SDEs were given in \cite{Higham-Mao-Stuart2003},\cite{Kloeden-Platen1991}. 
Despite the importance both on the theoretical and applied aspects of the random periodic regime, its numerical analysis has barely been 
developed. The only result we know is the pathwise approximations of the random periodic solutions of SDEs discussed in \cite{Feng-Liu-Zhao2017}. 
In this paper we study the weak approximation to periodic measures. 

We consider the following non-autonomous stochastic differential equations on $\R^d$
\begin{align}\label{equations of SDE}
	dX_t=b(t,X_t)dt+\sigma(X_t) dW_t,\quad t\geq s,
\end{align}
with initial condition $X_s=x$ where $b:\R \times\R^{d}\rightarrow\R^{d},  \sigma:\R^{d}\rightarrow\R^{d\times d}$, 
$W_t$ is a two-sided Wiener process in $\R^d$ on the Wiener probability space $(\Omega,\mathcal{F},\mathbb{P})$. 
We assume that $b$ is $\tau$-periodic in the time variable and weakly dissipative in the space variable. For a technical reason, here we only consider the case when $\sigma$ is time independent as we use the results in  \cite{Feng-Zhao-Zhong2019measure},\cite{Hopfner-Locherbach-Thieullen2017}. Denote 
by $X_t^{s,x}$ the solution of (\ref{equations of SDE}) throughout the paper.

The existence of the periodic measure was studied in \cite{Feng-Zhao-Zhong2019measure}. Under the assumption that the drift term is weakly dissipative and the diffusion term is non-degenerate, it was proved that the periodic measure $\rho_s$ exists and has a density function, denoted by $q(s,x)$. To obtain $q(s,x)$, one could solve the corresponding infinite horizon Fokker-Planck equation ${\frac{\partial}{\partial s}}q(s,x)=\tilde{\cL}^*(s) q(s,x)$ with additional condition $q(s+\tau,x)=q(s,x)$, where $\tilde{\cL}^*(s)$ is the adjoint of the infinitesimal generator of the process $X_t$. This partial differential equation is generally difficult to solve explicitly. But we will show that theoretically it plays an essential role in establishing the theory of numerical schemes of weak approximations.

We apply numerical schemes such as Euler-Maruyama method to estimate the periodic measure. For any fixed $i\in\Z$, denoted by $\{\hx_{i\dt+n\dt}^{i\dt,x}\}_{n=0,1,\cdots}$ the discrete approximation of the solution of (\ref{equations of SDE}) with step size $\dt=\frac{\tau}{N}$ and $\hx_{i\dt}^{i\dt,x}=x$.  We prove that the discrete semi-flow is geometrically ergodic and has a periodic measure $\hat\rho_i^\dt(\cdot)$, $i\in\Z$. 

In this paper, we use the idea of lifting the flow and periodic measure to the cylinder $[0,\tau)\times\R^d$ proposed in \cite{Feng-Zhao2018}. With the help of this tool, our main result is to prove that the cumulation of discretization errors is of the order of $\mathcal{O}(\dt)$ for the approximation of the average of periodic measure, i.e. for any ${\phi}\in C_p^\infty\left(\R^d\right)$
\begin{align}\label{1.2}
	\abs{ \int_{\R^d}\phi({x})\bar{\rho}(d{x})- \int_{\R^d}\phi({x})\bar{\hat \rho}^{\dt}(d{x}) }\leq C\dt,
\end{align} 
where $\bar{\rho} := \frac{1}{\tau}\int_0^{\tau}{\rho_s} ds$ (\cite{Feng-Zhao2018}), $\bar{\hat \rho}^\dt:=\frac{1}{N}\sum_{i=0}^{N-1}\hat\rho_i^\dt$, $C_p^\infty\left(\R^d\right)$ is the space of smooth functions with the property that themselves and all their derivatives have at most polynomial growth at infinity. In fact, (\ref{1.2}) only holds for $\dt$ being small enough and the choice of step size can be uniform for all $\phi\in C_p^\infty.$ For this, the uniformity of the step size working for all moment estimates of $\hx_{s+n\dt}^{s,x}$ is derived in Proposition \ref{proposition of bound of n-th moments of discrete process for weak approximation}. The error estimate (\ref{1.2}) can also be numerically verified.

The results in this paper are applicable for many physically relevant SDEs, for instance, Benzi-Parisi-Sutera-Vulpiani's stochastic resonance model for the ice-age transition in climate change dynamics is SDE 
(\ref{equations of SDE}), with $b(t,x)=x-x^3+A\cos (B t)$ and $\sigma (x)=\sigma$ being constant (\cite{Benzi-Parisi-Sutera-Vulpiani1981}). It was proved that this model has a unique periodic measure  (\cite{Feng-Zhao-Zhong2019measure}). This result implies the transition between 
ice-age and interglacial climates. A partial differential equation for expected transition time was given as well (\cite{Feng-Zhao-Zhong2019resonance}). This paper gives the weak approximation of numerical scheme for the SDE (\ref{equations of SDE})
with a modified drift which is nearly the same as the above $b$ when $x\in [-4,4]$ and linear when $x$ is far from this interval. This modified model provides the same climate dynamics as that of the original one of Benzi-Parisi-Sutera-Vulpiani since the global earth temperature cannot be outside of $[265, 305]$ in Kelvin scale. 

We first study the lifts of semi-flows and corresponding Fokker-Planck equation for the density of the periodic measure. The infinitesimal generator $\tilde{\cL}$ does not satisfy the non-degeneracy property with respect to initial time $s$. Under the weakly dissipative condition, we then obtain the exponential contraction of initial distribution to the periodic measure and all its spatial derivatives in the average with respect to initial time $s$. Finally, the numerical analysis on the cumulation of discretization errors is derived from these estimates and numerical experiments of error analysis are carried out for some specific SDEs arising in climate dynamics.

\section{Preliminary results and notation}
\subsection{Lifts of semi-flows, random periodic paths and periodic measures}
Denote by $(\Omega,\mathcal{F},\mathbb{P},(\theta_s)_{s\in\R})$ the metric dynamical system associated with the canonical probability space $(\Omega,\mathcal{F},\mathbb{P})$ for Brownian motion $W$ in $\R^d$, where $\theta_s:\Omega\to\Omega$ defined by $(\theta_s\omega)(t)=W(t+s)-W(s)$, is measurably invertible for all $s\in\R$. Denote $\Delta:=\{(t,s)\in\R^2,s\leq t\}$ and let $u:\Delta\times\Omega\times\R^d\to\R^d$ be a periodic stochastic semi-flow of period $\tau$ satisfying for all $(t,s)\in \Delta$ and $r\in [s,t]$
$$u(t,r,\omega)\circ u(r,s,\omega)=u(t,s,\omega), \quad {\rm for \ almost \ all}\ \omega \in \Omega,$$  
and
$$u(t+\tau,s+\tau,\omega)=u(t,s,\theta_{\tau}\omega), \quad {\rm for \ almost \ all}\  \omega \in \Omega.$$
Here $\tau>0$ is a deterministic real number. Solutions of stochastic differential equations (\ref{equations of SDE}) with coefficients being periodic in time with period $\tau$, when they exist and unique, generates a periodic semi-flow $u(t,s)x=X_{t}^{s,x}$, which satisfies
the above two properties. As we consider periodic measures in this paper, so perfection is not needed here.

Consider the case when $u(t+s,s,\cdot)$ is a Markovian semi-flow on a filtered dynamical system $(\Omega,\mathcal{F},\mathbb{P},(\theta_t)_{t\in\R},(\mathcal{F}_s^t)_{s\leq t}),$ i.e. for any $s,t,r\in\R,\ s\leq t$, we have $\theta^{-1}_r\mathcal{F}_s^t = \mathcal{F}_{s+r}^{t+r}$ and $u(t+s,s,\cdot)$ is independent of $\mathcal{F}_{-\infty}^s$, where $\mathcal{F}_{-\infty}^s:=\bigvee_{r\leq s}\mathcal{F}_r^s.$ For any $\Gamma\in\mathcal{B}(\R^d)$, $t\in\R^+$, $s\in\R$, denote the transition probability of $u$ by
$P(t+s,s,x,\Gamma)=\mathbb{P}(\{ \omega:u(t+s,s,\omega)x\in\Gamma \}).$
From the periodicity of semi-flow $u$ and the measure preserving property of $\theta_\tau$, the transition probability $P(t+s,s,x,\cdot)$ satisfies the periodic relation
\begin{align}\label{periodic property of P}
	P(t+s+\tau,s+\tau,x,\cdot)=P(t+s,s,x,\cdot).
\end{align}
Define for $\phi\in\mathcal{B}_b(\R^d)$, the space of bounded and Borel measurable function from $\R^d$ to $\R$,
$$
\mathcal{T}(t+s,s)\phi(x):=\E\phi(X_{t+s}^{s,x})=\int_{\R^d}\phi(y)P(t+s,s,x,dy),\ t\geq 0.
$$
Then it is well-known that $\mathcal{T}(t+s,s):\mathcal{B}_b(\R^d)\to\mathcal{B}_b(\R^d)$ defines a semigroup and satisfies the $\tau$-periodic property:
$$\mathcal{T}(t+s+\tau,s+\tau)=\mathcal{T}(t+s,s).$$
This follows from (\ref{periodic property of P}) and the definition of $\mathcal{T}(t+s,s)$ easily.
Moreover for any probability measure $\rho\in\mathcal{P}(\R^d)$, the space of probability measures on $({\R^d},\mathcal{B}({\R^d})),$ define
$$(\mathcal{T}^*(t+s,s)\rho)(\Gamma)=\int_{{\R^d}}P(t+s,s,x,\Gamma)\rho(dx).
$$
The definition of periodic measure of the periodic Markovian semi-group is given as follows. 
The existence of the periodic measure was proved in \cite{Feng-Zhao-Zhong2019measure} for a wide class of SDEs.

\begin{definition}(\cite{Feng-Zhao2018})
	The measure valued function $\rho:\R\to\mathcal{P}({\R^d})$ is called a $\tau$-periodic measure of the $\tau$-periodic Markovian semi-group $\mathcal{T}$ if
	\begin{align}
		\mathcal{T}^*(t+s,s)\rho_s=\rho_{t+s},\ \rho_{s+\tau}=\rho_s, \quad \forall s\in\R,\ t\in\R^+.
	\end{align}
\end{definition}


The idea of lifting a stochastic periodic semi-flow to a cocycle on a cylinder in \cite{Feng-Zhao2018} plays an important role in this paper.  As for this paper, the relevant part is briefly discussed below. 
Let $\mathbb{S}=[0,\tau)\times\R^d$, the lifted cocycle arising from SDE (\ref{equations of SDE}) with coordinates $\tilde{X}_s=(s,X_s)\in\mathbb{S}$ is given by 
\begin{align*}
	d\tilde{X}_t=\tilde{b}(\tilde{X}_t)dt+\tilde{\sigma}(\tilde{X}_t)d\tilde{W}(t),
\end{align*}
where
$\tilde{X}_0^{0,\tilde{x}}=\tilde{x}=(s,x),\ \tilde{W}=(\tilde{W}_0,W),$
$\tilde{W}_0$ is a one-dimensional Brownian motion which is independent of $W$,
$
\tilde{b}(\tilde{X}_t)=\left(
\begin{array}{c}
	1\\
	b(t,X_t)
\end{array}
\right),\quad \tilde{\sigma}(\tilde{X}_t)=\left(
\begin{array}{cc}
	0 & 0\\
	0 & \sigma(X_t)
\end{array}
\right).
$
One can enlarge the probability space $(\Omega,\mathcal{F},\mathbb{P})$, still denoted by $(\Omega,\mathcal{F},{\mathbb{P}})$, as the canonical probability space for $\R^{d+1}$ Brownian motion $\tilde{W}$.

It is easy to see that the infinitesimal generator of the process $\tilde X$ is given by
\begin{align}\label{definition of infinitesimal generator}
	\tilde{\cL}&=\sum_{i=1}^{d} b_i(s,x)\frac{\partial}{\partial x_i}+\frac{1}{2} \sum_{i,j=1}^{d} a_{ij}(x)\frac{\partial^2}{\partial x_i\partial x_j}+
	\frac{\partial}{\partial s}=: \cL(s)+\frac{\partial}{\partial s},
\end{align}
where $a=(a_{ij})=\sigma\sigma^T$, and $\mathcal{U}(t+s,s,x):=\mathcal{T}(t+s,s)\phi(x)$ satisfies
\begin{align}\label{equation of lifted U}
	\frac{\partial}{\partial t}\mathcal{U}(t+s,s,x) = \tilde{\cL}\mathcal{U}(t+s,s,x),\quad \mathcal{U}(s,s,x)= \phi(x),
\end{align}
provided $\mathcal{U}$ is sufficiently smooth.
Meanwhile, the transition probability and the periodic measure are lifted to
\begin{align*}
	\tilde{\rho}_s(C\times\Gamma)=&\delta_{(s\ mod\ \tau)}(C)\rho_s(\Gamma),\\
	\tilde{P}(t,\tilde{x},C\times\Gamma)=&\delta_{(t+s\ mod\ \tau)}(C)P(t+s,s,x,\Gamma)=\mathbb{P}(\{\omega:X_{t+s}^{s,x}\in\Gamma\}),
\end{align*}
where $C\in\mathcal{B}([0,\tau))$ and $\Gamma\in\mathcal{B}({\R^d}).$ It was shown in \cite{Feng-Zhao2018} that  $\tilde{P}$ generates a homogeneous semigroup $\tilde{\mathcal{T}}$ defined by 
$(\tilde{\mathcal{T}}\tilde{\phi})(\tilde{x})=\int_{\mathbb{S}}\tilde{\phi}(\tilde{y})\tilde{P}(t,\tilde{x},d\tilde{y})$
and $\tilde{\rho}_s$ is a periodic measure of the lifted semigroup $\tilde{\mathcal{T}}$. It was also noticed that $\apm=\frac{1}{\tau}\int_{0}^{\tau}\tilde{\rho}_sds$ is an invariant measure of $\tilde{P}$ following a standard procedure of Fubini theorem. It is easy to see that for a measurable function $\phi:\R^d\to\R$,
\begin{align}\label{equivalence of measurable function under lifted and original apm}
	\int_0^\tau\int_{\R^d}\phi(x)\apm(dt,dx)=&\frac{1}{\tau}\int_0^\tau\int_0^\tau\int_{\R^d}\phi(x)\delta_{(s\ mod\ \tau)}(dt)\rho_s(dx)ds\\
	=&\frac{1}{\tau}\int_0^\tau\int_{\R^d}\phi(x)\rho_s(dx)ds\notag\\
	=&\int_{\R^d}\phi(x)\bar\rho(dx),\notag
\end{align} 
where $\bar\rho=\frac{1}{\tau}\int_0^\tau\rho_sds.$ This will be used in later part of this paper.

%

\subsection{Assumptions and some preliminary estimates}


Assume

{\bf Condition (1)}
{\it The functions $b$, $\sigma$ are of class $C^\infty$ with $\sigma$ being bounded, $b$ and $\sigma$ having bounded derivatives of any order and $b$ being $\tau$-periodic with respect to time. 
}
\vskip5pt


{\bf Condition (2)}
{\it (Uniform ellipticity) There exists a positive constant $\alpha$ such that for any $x,y\in \R^d$, we have
	$
	\sum_{i,j}a_{ij}(y)x_i x_j\geq \alpha\abs{x}^2.
	$
}
\vskip5pt

%

{\bf Condition (3)}
{\it (Weak dissipativity) There exist constants $\beta>0$ and $C>0$ such that for any $t\in\R^+$ and any $x\in\R^d$,
	$x\cdot b(t,x)\leq -\beta\abs{x}^2+C.$}
\vskip5pt

Under conditions (1)-(3), it was proved in \cite{Feng-Zhao-Zhong2019measure} that the periodic measure $\rho:(-\infty,+\infty)\to \mathcal{P}(\R^d)$ exists and is geometrically ergodic:
$$\norm{P(n\tau+s,s,x)-\rho_s}_{TV}\leq Ce^{-\delta n\tau}.$$ 
We now discuss the existence of the density function $q(s,x)$ of the periodic measure $\rho_s$. Set the Fokker-Planck operator as follows
$$\tilde{\cL}^*(s)\cdot=-\sum_{i=1}^{d} \frac{\partial }{\partial x_i}\left( b_i(s,x)\cdot\right)+\frac{1}{2} \sum_{i,j=1}^{d} \frac{\partial^2 }{\partial x_i\partial x_j}\left( a_{ij}(x)\cdot\right),$$
and
$$\tilde{\cL}^*=\tilde{\cL}^*(s)-\frac{\partial}{\partial s}.$$

\begin{proposition}\label{proposition of existence of density function of periodic measure}
	Assume Conditions (1), (2) and (3). Then the periodic measure $\rho_s$ has a density $q(s,x)$ with respect to the Lebesgue measure in $\R^d$, and the density is the unique bounded solution of the Fokker-Planck equation
	\begin{align}\label{equation of fokker-planck for density of pm}
		\frac{\partial }{\partial s}q(s,x)=\tilde{\cL}^*(s)q(s,x),
	\end{align}
	satisfying that for any $s\in[0,\tau)$, $q(s+\tau,x)=q(s,x)$
	and $q(s,x)\to 0$ as $\abs{x}\to\infty.$
\end{proposition}
\begin{proof}
	Under the assumption of this proposition, the $\tau$-periodic two-parameter Markov transition probability $P(t,s,x,\Gamma)$ has a density $p(t,s,x,y).$ 
	Thus we have the representation of periodic measure $\rho_s$ as follows, for any $\Gamma\in\mathcal{B}(\R^d)$,
	\begin{align*}
		\rho_s(\Gamma) 
		=&\int_{\R^d}\int_{\Gamma}p(s+\tau,s,x,y)dy\rho_s(dx)=\int_{\Gamma}\int_{\R^d}p(s+\tau,s,x,y)\rho_s(dx)dy,
	\end{align*}
	where we applied Fubini's theorem.
	Hence we get the formula of the density of $\rho_s$ as
	\begin{align}\label{equation of definition of density function by periodic measure}
		q(s,y) = \int_{\R^d}p(s+\tau,s,x,y)\rho_s(dx).
	\end{align}
	It is easy to prove the periodicity of the density $q(s,y)$ by the periodic property of both $p(t,s,x,y)$ and $\rho_s$.
	Moreover, we have that for any $\Gamma\in\mathcal{B}(\R^d)$,
	\begin{align*}
		\rho_{t+\tau}(\Gamma)
		=&\int_{\Gamma}\int_{\R^d} p(t+\tau,s+\tau,z,y)\int_{\R^d}p(s+\tau,s,x,z)\rho_s(dx)dzdy\\
		=&\int_{\Gamma}\int_{\R^d} p(t,s,z,y)q(s,z)dzdy.
	\end{align*}
	As the periodic measure satisfies $\rho_t(\Gamma)=\rho_{t+\tau}(\Gamma)$, the above implies
	\begin{align}\label{equation of definition of density function by transition}
		q(t,y)=\int_{\R^d}p(t,s,z,y) q(s,z) dz.
	\end{align}
	It is well known that $p(t,s,y,x)$ satisfies the Fokker-Planck equation
	$\partial_t p(t,s,y,x)=\tilde{\cL}^*(t)p(t,s,y,x)$.
	Therefore,
	\begin{eqnarray*}
		\partial_t q(t,x) 
		&=& \int_{\R^d}\tilde{\cL}^*(t)p(t,s,y,x) q(s,y) dy\\
		&=& -\sum_{i=1}^{d} \frac{\partial }{\partial x_i}\left( b_i(t,x)\int_{\R^d}p(t,s,y,x) q(s,y) dy \right)\\
		&&+\frac{1}{2} \sum_{i,j=1}^{d} \frac{\partial^2 }{\partial x_i\partial x_j}\left( a_{ij}(x)\int_{\R^d}p(t,s,y,x) q(s,y) dy\right)\\
		&=& \tilde{\cL}^*(t)\int_{\R^d}p(t,s,y,x)q(s,y) dy= \tilde{\cL}^*(t)q(t,x),
	\end{eqnarray*}
	which implies the density $q(s,y)$ satisfies the equation (\ref{equation of fokker-planck for density of pm}).
	The claim that $q(t,y)\to 0$ as $y\to\infty$ follows from (\ref{equation of definition of density function by transition}) and the fact that when $\abs{y}\to \infty$, we have $p(t,s,z,y)\to 0.$  
\end{proof}

\begin{corollary}\label{corollary of L^*q(s,x)=0}
	If the density function $q(s,x)$ of periodic measure satisfies the equation (\ref{equation of fokker-planck for density of pm}), then for any $\tau$-periodic function $f\in C_p^\infty$, we have 
	$$\int_0^\tau\int_{\R^d}\tilde{\cL} f(s,x)q(s,x)dxds=0.$$
\end{corollary}
\begin{proof}
	The main ingredient of proof is to apply integration by parts. 
	Note first
	\begin{align*}
		\int_0^\tau\int_{\R^d} b_i(s,x)\frac{\partial}{\partial x_i} f(s,x)q(s,x)dxds=&-\int_0^\tau\int_{\R^d} f(s,x) \frac{\partial}{\partial x_i}\left(b_i(s,x)q(s,x)\right)dxds,
		\\
		\int_0^\tau\int_{\R^d}  a_{ij}(x)\frac{\partial^2}{\partial x_ix_j} f(s,x)q(s,x)dxds
		=&\int_0^\tau\int_{\R^d}  f(s,x) \frac{\partial^2}{\partial x_ix_j}\left(a_{ij}(x)q(s,x)\right)dx ds.
	\end{align*}
	Here we used the property that $q(s,x)$ vanishes as $x$ goes to $\infty$ when we performed the integration by parts. Applying the periodicity with respect to time $s$ of function $f$ and density function $q$ in the third part, we have
	\begin{align*}
		\int_0^\tau\int_{\R^d} \frac{\partial}{\partial s} f(s,x)q(s,x)dxds=-\int_0^\tau\int_{\R^d} f(s,x) \frac{\partial}{\partial s}q(s,x)dxds.
	\end{align*}
	Therefore, by the Fokker-Planck equation on the density function $q(s,x)$, we have
	\begin{align*}
		\int_0^\tau\int_{\R^d}\tilde{\cL} f(s,x)q(s,x)dxds=\int_0^\tau\int_{\R^d}f(s,x)\tilde{\cL^*} q(s,x)dxds=0.
	\end{align*}
\end{proof}
%
\begin{proposition}\label{proposition of boundedness of n-th moments for exact solution}
	Assume Conditions (1) and (3). Then for any $p\in\N$, there exist strictly positive constants $C_p$ and $\gamma_p$, such that for any $t>0$ and $x\in\R^d$,
	\begin{align}\label{inequality of p-th power of exact solution bounded}
		\E\abs{X_{t+s}^{s,x}}^p\leq C_p(1+\abs{x}^p \exp(-\gamma_p t)).
	\end{align}
\end{proposition}
\begin{proof}
	Denote by $X_t := X_{t+s}^{s,x}$ for simplicity.
	Applying It\^o's formula and Conditions (1), (3), we have the estimate
	\begin{align*}
		d \left( e^{\delta t}\abs{X_t}^p \right)
		\leq& \left(\delta-p\beta\right) e^{\delta t}\abs{X_t}^{p} dt + p\sigma(X_t) e^{\delta t}\abs{X_t}^{p-1}dW_t\\
		&+\left(pC\sigma+\binom{p}{2} C_\sigma^2\right) e^{\delta t}\abs{X_t}^{p-2}dt,
	\end{align*}
	where $C_\sigma$ is the bound of function $\sigma$, $C$ and $\beta$ are the constants in the weakly dissipative condition. For convenience, here we denote $C_{p,\sigma} = pC+\binom{p}{2} C_\sigma^2$. Let $\tau_N$ be the first exit time of the process $X_t$ from the ball of radius $N$. Consider the expectation of the integral $\E\int_0^{T\wedge \tau_N} \abs{X_t}^pdW_t=0$ for arbitrary $p$. Now take expectation on both sides after integrating from 0 to $T\wedge\tau_N$, together with Young's inequality, we have
	\begin{align*}
		&\E e^{\delta (T\wedge\tau_N)}\abs{X_{T\wedge\tau_N}}^p \\
		\leq& \abs{x}^p  + \left(\delta-p\beta\right)\E\int_0^{T\wedge\tau_N} e^{\delta t}\abs{X_t}^{p} dt + C_{p,\sigma}\E \int_0^{T\wedge\tau_N} e^{\delta t}\abs{X_t}^{p-2}dt\\
		\leq&\abs{x}^p   + \frac{2C_{p,\sigma}}{p\delta\varepsilon^{\frac{p}{2}}}\E  (e^{\delta (T\wedge\tau_N)}-1)+ K_1\E
		\int_0^{T\wedge\tau_N} e^{\delta t}\abs{X_t}^{p} dt,
	\end{align*}
	where
	$K_1 = \delta-p\beta+\frac{(p-2)C_{p,\sigma}}{p}\varepsilon^{\frac{p}{p-2}}, \ \varepsilon<\rbrac{    \frac{p^2\beta}{(p-2)C_{p,\sigma}}    }^{\frac{p-2}{p}}$
	is chosen such that $K_1 - \delta<0$. The choice of the constant $\delta$ guarantees  $K_1>0$.
	\begin{align*}
		\abs{X_t}^{p-2}\leq \frac{(\abs{X_t}^{p-2}\varepsilon)^{\frac{p}{p-2}}}{\frac{p}{p-2}}+\frac{\rbrac{\frac{1}{\varepsilon}}^{\frac{p}{2}}}{\frac{p}{2}}=\frac{p-2}{p}\varepsilon^{\frac{p}{p-2}}\abs{X_t}^p+\frac{2}{p\varepsilon^{\frac{p}{2}}}.
	\end{align*}
	Then we let $N$ go to $\infty$ to obtain
	\begin{align}\label{inequality for n-th moment bounded of weak approximation}
		e^{\delta T}\E\abs{X_T}^p
		\leq& \abs{x}^p   + \frac{2C_{p,\sigma}}{p\delta\varepsilon^{\frac{p}{2}}}  (e^{\delta T}-1)+ K_1
		\int_0^T e^{\delta t}\E\abs{X_t}^{p} dt.
	\end{align}
	Apply Gronwall's inequality on (\ref{inequality for n-th moment bounded of weak approximation}),
	\begin{align}\label{inequality of estimate of pth moment}
		e^{\delta T}\E\abs{X_T}^p
		\leq\frac{2C_{p,\sigma}}{p\delta\varepsilon^{\frac{p}{2}}} e^{\delta T}  +e^{K_1T}\left( \abs{x}^p - \frac{2C_{p,\sigma}}{p\delta\varepsilon^{\frac{p}{2}}}  \right) + \frac{2K_1 C_{p,\sigma}}{p(\delta-K_1)\delta\varepsilon^{\frac{p}{2}}} (e^{\delta T} - e^{K_1 T}).
	\end{align}
	Then (\ref{inequality of p-th power of exact solution bounded}) follows easily.
\end{proof}
\begin{proposition}\label{proposition of boundedness of nth moments under periodic measure}
	Assume Conditions (1) and (3). Then for any $p\in\N,$
	$\frac{1}{\tau}\int_0^\tau\int_{\R^d}\abs{x}^p q(s,x)dxds\leq C_p,$
	where $C_p$ is determined from Proposition \ref{proposition of boundedness of n-th moments for exact solution}. 
\end{proposition}
\begin{proof}
	For the density function of transition kernel $P(t+s,s,x,\cdot)$, there exists a constant $C$ such that 
	$\abs{p(t+s,s,x,y)}\leq C,$
	for any $t\geq 1.$
	Then by dominated convergent theorem and Theorem 3.7 in \cite{Feng-Zhao-Zhong2019measure}, for any compact set $K\subset\R^d$, we have
	\begin{align*}
		&\frac{1}{\tau}\int_0^{\tau}\lim_{n\to\infty}\E\left(\abs{X_{n\tau+s}^{s,x}}^p 1_K(X_{n\tau+s}^{s,x})\right)ds=\frac{1}{\tau}\int_0^{\tau}\lim_{n\to\infty}\int_{K}\abs{y}^p p(n\tau+s,s,x,y)dy ds\\
		=&\frac{1}{\tau}\int_0^{\tau}\int_{K}\abs{y}^p\lim_{n\to\infty} p(n\tau+s,s,x,y)dyds
		=\frac{1}{\tau}\int_0^{\tau}\int_{K}\abs{y}^pq(s,y)dy ds.
	\end{align*}
	Thus the average of periodic measure possesses finite moments of any order on any compact set $K$ from the estimates in Proposition \ref{proposition of boundedness of n-th moments for exact solution}. Note the bound can be independent of $K$. The result follows from taking limit $K\uparrow\R^d$ and Fatou's Lemma.
\end{proof}

Consider the sequence $\{X_{t_n}\}_{n\in\N}$ with $t_n=n\tau$. We prove 
\begin{proposition}\label{proposition of geometrical recurrent condition}
	Assume Conditions (1) and (3), then there exists a constant $r>1$ and a ball $B(0,R)$, such that,
	\begin{align*}
		\sup_{x\in B(0,R)^c}\E\left[ r\abs{X_{t_{n+1}}}^2 - \abs{X_{t_n}}^2 \bigg\lvert X_{t_n}=x  \right]<0.
	\end{align*}
\end{proposition}
\begin{proof}
	Apply the result (\ref{inequality of estimate of pth moment}) in Proposition \ref{proposition of boundedness of n-th moments for exact solution} with $p=2$, we have
	\begin{align*}
		\E\left[ r\abs{X_{t_{n+1}}}^2 - \abs{X_{t_n}}^2 \bigg\lvert X_{t_n}=x  \right] 
		\leq& \left( \abs{x}^2 e^{(K_1-\delta)\tau} + C_{\beta,\sigma} (1-e^{(K_1-\delta)\tau})\right)r-\abs{x}^2.
	\end{align*}
	In order to make the right hand side of the above negative, we need
	$(1-re^{(K_1-\delta)\tau})\abs{x}^2\linebreak
	> r C_{\beta,\sigma} (1-e^{(K_1-\delta)\tau}).$
	As $K_1-\delta<0$, there always exists a constant $r$ to ensure $1-re^{(K_1-\delta)\tau}>0$ for the given period $\tau$. Then the ball $B(0,R)$ is determined by taking $R>\sqrt{\frac{r C_{\beta,\sigma} (1-e^{(K_1-\delta)\tau})}{1-re^{(K_1-\delta)\tau}}}$.
\end{proof}

Let the function $\phi\in C_p^{\infty}$ and $\mathcal{U}(t+s,s,x) = \E\phi(X_{t+s}^{s,x})$. Then $\mathcal{U}$ satisfies the PDE (\ref{equation of lifted U}).
Considering the spatial differentiation of the solution with respect to $x$, Kunita showed in \cite{Kunita1984} that the function $\mathcal{U}(t+s,s,x)$ satisfies that for any order $n\in\N$, there exists an integer $r_n\in\N$ such that for any $T>0$, $\exists C_n(t)>0$,
\begin{align}\label{inequality of kunita's estimation}
	\abs{D^n \mathcal{U}(t+s,s,x)}\leq C_n(T)(1+\abs{x}^{r_n}),\ \forall t<T.
\end{align}
From Proposition \ref{proposition of boundedness of nth moments under periodic measure}, the average of periodic measure possesses finite moments of any order. Together with (\ref{inequality of kunita's estimation}), we have that the initial condition $\phi$ and $D^n\mathcal{U}(t+s,s,x)$ belong to $L^2(\R^{d+1},\apm).$

Note that the function $\frac{1}{\tau}\int_0^{\tau} \mathcal{U}(t+s,s,x)ds$ has the same spatial derivatives as
$\frac{1}{\tau}\int_0^{\tau}\mathcal{U}(t+s,s,x)ds - \frac{1}{\tau}\int_0^\tau\int_{\R^d}\phi(x)q(s,x)dxds.$
Without loss of generality, in the following sections, we assume that
\begin{align}\label{equality of annihilation of initial condition}
	\int_0^\tau\int_{\R^d}\phi(x)q(s,x)dx ds=0.
\end{align}
Note when $\tilde{\phi}(\tilde{x})=\phi(x)$, we have $\int_{\mathbb{S}}\tilde{\phi}(\tilde{x})d\apm(\tilde{x}) = \frac{1}{\tau}\int_0^\tau\int_{\R^d}\phi(x)q(s,x)dxds,$ where $\apm$ is the average of lifted periodic measure, which is the invariant measure of the lifted Markov semigroup. It is easy to know that $\apm(d\tilde{x})=q(s,x)dxds$

For simplicity, in the following sections, we may often write $\tilde{\mathcal{U}}(t)$ or $\mathcal{U}(t+s,s)$ to represent the function $\mathcal{U}(t+s,s,x)$. 
We also often write $b$ to represent $b(s,x)$ as we have the uniform conditions for the function and any order of its derivatives in Condition (1). 
The operators $\partial_i$, $\partial_{ij}$, $\nabla$ and $D^k$ on function $\mathcal{U}(t+s,s,x)$ always refer to derivatives with respect to spatial coordinates. The derivatives with respect to initial time will stay as $\frac{\partial}{\partial s}$. 

\section{Exponential decay of initial distribution and spatial derivatives}
\subsection{Estimates on the average of $\mathcal{U}(t+s,s)$ on a ball} \label{section 3.1}
We always assume (\ref{equality of annihilation of initial condition}) in this section unless otherwise stated.
\begin{lemma}\label{lemma of exponential contraction of u for any Ball}
	Assume Conditions (1), (2) and (3). Then for any ball $B$, there exist strictly positive constants $C$ and $\lambda$ such that for any $t>0$ and any $x\in B,$ $\mathcal{U}$ defined with $\phi$ satisfying (\ref{equality of annihilation of initial condition}) has the following estimate:
	$\frac{1}{\tau}\int_0^{\tau}\abs{\mathcal{U}(t+s,s,x)}ds\leq C\exp(-\lambda t).$
\end{lemma}
\begin{proof}
	First we apply mathematical induction to obtain that for any $p\in\N^+$, there exist constants $C_p>0,$ $\gamma_p>0$ such that
	\begin{align}\label{exponential contraction of D^pU}
		\frac{1}{\tau}\int_0^{\tau}\int_{\R^d}\abs{D^p \mathcal{U}(t+s,s,x)}^2 q(s,x) dxds \leq C_p \exp(-\gamma_p t).
	\end{align} 
	We start to prove the basis step, when $p=1$.
	Consider the Markov chain $\{X_{t_n}\}_{n\in\N}$ with $t_n=n\tau$. In \cite{Feng-Zhao-Zhong2019measure}, it was proved that the transition kernel $P(s,s+k\tau,x,\Gamma)$ is irreducible. With the result of Proposition \ref{proposition of geometrical recurrent condition}, one can find some compact set $K$ and a constant $\beta>0$ such that for any $x\in K^c$, we have
	$$\E \abs{X_t^{0,x}}^2 - \frac{1}{r} \E x^2<0,$$
	where $1/r<1.$ Now we take $V(x)=x^2$ as the norm-like function and from Proposition \ref{proposition of boundedness of n-th moments for exact solution}, we obtain that the norm-like function $V(x)=x^2$ is finite on the compact set $K$. Combining the above results, we have that
	$$(P(t_1,0)V)(x)=\E X_{t_1}^2\leq \frac{1}{r} x^2+\beta=\frac{1}{r} V(x)+\beta,$$
	where $\beta$ is a positive number. Thus the condition of Theorem 3.7 in \cite{Feng-Zhao-Zhong2019measure} is satisfied. So the Markov chain $\{X_{t_n}\}_{n\in\N}$ is geometrically ergodic.
	That is for those $\phi\in C_p^{\infty}$ with the assumption (\ref{equality of annihilation of initial condition}), there exist strictly positive constants $C$ and $\lambda$ such that for any $n$,
	\begin{align}\label{inequality of exponential contraction on period points}
		\frac{1}{\tau}\int_0^\tau\int_{\R^d}  \abs{\E \phi(X_{t_n+s}^{s,x}) }   q(s,x)dx ds \leq Ce^{-\lambda t_n}.
	\end{align}
	As function $\phi$ has at most polynomial growth at infinity, we have $\abs{\phi(x)}\leq C\abs{x}^N$ for some integer $N\in\N$. 
	By Proposition \ref{proposition of boundedness of n-th moments for exact solution}, there exist $C_0>0$, $\gamma>0$ such that
	\begin{align}\label{inequality of boundedness of semigroup}
		\abs{\mathcal{U}(t+s,s,x)}\leq C_0(1+\abs{x}^N\exp(-\gamma t)).
	\end{align}
	Applying (\ref{inequality of boundedness of semigroup}) and (\ref{inequality of exponential contraction on period points}), together with Proposition \ref{proposition of boundedness of nth moments under periodic measure}, we have that for any $n$,
	\begin{align}\label{inequality of tn exponential contraction}
		&\hspace*{0.3cm}\frac{1}{\tau}\int_0^\tau\int_{\R^d}  \abs{\mathcal{U}(t_n+s,s,x)}^2   q(s,x)dx ds\\
		\leq& \frac{C_0}{\tau}\int_0^\tau\int_{\R^d}  \abs{\mathcal{U}(t_n+s,s,x)} (1+\abs{x}^N\exp(-\gamma t_n))  q(s,x)dx ds\notag\\
		\leq& C_0C\exp(-\lambda t_n) + \frac{C_0^2\exp(-\gamma t_n)}{\tau}\int_0^\tau\int_{\R^d}   \abs{x}^N  q(s,x)dx ds\notag\\
		\leq& C_1\exp(-\lambda_1 t_n)\notag.
	\end{align}
	In the following, we prove that the function $t\to \frac{1}{\tau}\int_0^\tau\int_{\R^d}  \abs{\mathcal{U}(t+s,s,x)}^2   q(s,x)dxds$ is monotonic. For this, note
	\begin{align*}
		\frac{d}{dt}\abs{\mathcal{U}(t+s,s,x)}^2= \tilde{\cL} \abs{\mathcal{U}(t+s,s,x)}^2 - a_{ij}\partial_{i} \mathcal{U}(t+s,s,x)\partial_{j} \mathcal{U}(t+s,s,x),
	\end{align*}
	and 
	{\small $$\frac{1}{\tau}\int_0^\tau\int_{\R^d}  \tilde{\cL} \abs{\mathcal{U}(t+s,s,x)}^2   q(s,x)dxds=\frac{1}{\tau}\int_0^\tau\int_{\R^d}   \abs{\mathcal{U}(t+s,s,x)}^2   \tilde{\cL}^*q(s,x)dxds=0.$$}
	It turns out from the elliptic condition (2) that
	\begin{align*}
		&\frac{d}{dt}\left(  \frac{1}{\tau}\int_0^\tau\int_{\R^d}  \abs{\mathcal{U}(t+s,s,x)}^2   q(s,x)dxds  \right)\\
		=&\frac{1}{\tau}\int_0^\tau\int_{\R^d}  \tilde{\cL} \abs{\mathcal{U}(t+s,s,x)}^2   q(s,x)dxds  \\
		&- \frac{1}{\tau}\int_0^\tau\int_{\R^d}  a_{ij}(x)\partial_{i} \mathcal{U}(t+s,s,x)\partial_{j} \mathcal{U}(t+s,s,x) q(s,x)dx ds\\
		\leq & - \frac{\alpha}{\tau}\int_0^\tau\int_{\R^d}  \abs{\nabla \mathcal{U}(t+s,s,x)}^2 q(s,x)dx ds\leq 0.
	\end{align*}
	This implies that $ \frac{1}{\tau}\int_0^\tau\int_{\R^d}  \abs{\mathcal{U}(t+s,s,x)}^2   q(s,x)dxds$ is decreasing in $t$. Thus by (\ref{inequality of tn exponential contraction}), 
	we have that for any $t$,
	\begin{align}\label{inequality of exponential contraction for general t}
		&\frac{1}{\tau}\int_0^\tau\int_{\R^d}  \abs{\mathcal{U}(t+s,s,x)}^2   q(s,x)dxds  \leq C_2\exp(-\lambda_1 t).
	\end{align} 
	The above shows that the exponential contraction of $\mathcal{U}(t+s,s,x)$ under the average of periodic measure holds for any $t$.
	
	On the other hand, by Condition (2) we have
	\begin{align}\label{inequality of comparison between L and dt with alpha}
		\frac{d}{dt}\abs{\mathcal{U}(t+s,s,x)}^2 - \tilde{\cL} \abs{\mathcal{U}(t+s,s,x)}^2 \leq -\alpha\abs{\nabla \mathcal{U}(t+s,s,x)}^2.
	\end{align}
	Multiplying the above inequality with $e^{\delta t}$, and integrating both sides with respect to the average periodic measure $\apm$ and time $t$, together with Corollary \ref{corollary of L^*q(s,x)=0}, we obtain for arbitrary $T>0,$
	\begin{align}\label{inequality of integration from 0 to T on u for weak approximation}
		&\int_0^T e^{\delta t}\int_\mathbb{S}    \frac{d}{dt}\abs{\mathcal{U}(t+s,s,x)}^2d\apm dt + \int_0^T \alpha e^{\delta t} \int_\mathbb{S} \abs{\nabla \mathcal{U}(t+s,s,x)}^2d\apm dt\\
		\leq&\int_0^T e^{\delta t}  \int_\mathbb{S}   \tilde{\cL} \abs{\mathcal{U}(t+s,s,x)}^2 d\apm dt = 0.\notag
	\end{align}
	Here ${\mathbb S}=[0,\tau)\times {\mathbb R}^d$. Integration by parts on the first term of (\ref{inequality of integration from 0 to T on u for weak approximation}) gives us
	\begin{align*}
		&\int_0^T e^{\delta t}  \int_\mathbb{S}   \frac{d}{dt}\abs{\mathcal{U}(t+s,s,x)}^2d\apm dt \\
		=&  e^{\delta T} \int_\mathbb{S}   \abs{\mathcal{U}(T+s,s,x)}^2d\apm - \int_\mathbb{S}  \abs{\mathcal{U}(s,s,x)}^2d\apm - \delta \int_0^T  e^{\delta t} \int_\mathbb{S}     \abs{\mathcal{U}(t+s,s,x)}^2d\apm dt,
	\end{align*}
	where we have the initial condition that $\mathcal{U}(s,s,x)=\tilde{\phi}(\tilde{x})$. By Proposition \ref{proposition of boundedness of n-th moments for exact solution} and $\phi\in C_p^{\infty}$ of the function $\phi$, we have a constant $C_3>0$ such that$
	\int_{\mathbb{S}}  \abs{\tilde{\phi}(\tilde{x})}^2d\apm<C_3.$
	Consider (\ref{inequality of exponential contraction for general t}) and take $\delta < \lambda_1$,
	\begin{align*}
		&\delta \int_0^T  e^{\delta t}  \int_\mathbb{S}   \abs{\mathcal{U}(t+s,s,x)}^2d\apm dt
		\leq\delta \int_0^T  e^{\delta t} C_2 e^{(-\lambda_1 t)} dt
		=\frac{C_2\delta}{\lambda_1-\delta}(1-e^{(\delta-\lambda_1)T})\leq C_4
	\end{align*}
	for a constant $C_4>0.$
	Applying these results to (\ref{inequality of integration from 0 to T on u for weak approximation}), we obtain that for any $T$ and any $s\in[0,\tau)$,
	\begin{align}\label{inequality of boundedness of first derivative of u under apm and time int}
		&\int_0^T  e^{\delta t} \int_\mathbb{S}  \abs{\nabla \mathcal{U}(t+s,s,x)}^2d\apm dt  \notag\\
		\leq& \frac{1}{\alpha} \left( \int_\mathbb{S}  \abs{\mathcal{U}(s,s,x)}^2d\apm + \delta \int_0^T   e^{\delta t} \int_\mathbb{S}   \abs{\mathcal{U}(t+s,s,x)}^2d\apm dt \right)\leq C_5,
	\end{align}
	where $C_5=C_3+C_4.$
	Now let's consider $\abs{\nabla \mathcal{U}(t+s,s,x)}^2$ and note that
	{\small \begin{align*}
			&\frac{d}{dt}\abs{\nabla \mathcal{U}(t+s,s,x)}^2 - \tilde{\cL}\abs{\nabla \mathcal{U}(t+s,s,x)}^2 \\
			=& - a_{ij}(\partial_{ik} \mathcal{U}(t+s,s,x))(\partial_{jk} \mathcal{U}(t+s,s,x)) + 2(\partial_k b_i)(\partial_k \mathcal{U}(t+s,s,x))(\partial_i \mathcal{U}(t+s,s,x)) \\
			&+ (\partial_k a_{ij}) (\partial_k \mathcal{U}(t+s,s,x))(\partial_{ij} \mathcal{U}(t+s,s,x)).
	\end{align*}}
	Applying Young's inequality with $\varepsilon$, we have
	\begin{align*}
		&(\partial_k a_{ij}) (\partial_k \mathcal{U}(t+s,s,x))(\partial_{ij} \mathcal{U}(t+s,s,x))\\
		\leq&\frac{\varepsilon}{2}(\partial_{ij}\mathcal{U}(t+s,s,x)))^2+\frac{(\partial_k a_{ij}\partial_k\mathcal{U}(t+s,s,x)))^2}{2\varepsilon}.
	\end{align*}
	From Conditions (1) and (2), we can choose $\varepsilon$ small enough such that $-\alpha+\frac{\varepsilon}{2}<0$. It turns out that there exist strictly positive constants $C_6$ and $C_7$ such that
	\begin{align}\label{inequality of comparison between L and dt with C}
		&\frac{d}{dt}\abs{\nabla \mathcal{U}(t+s,s,x)}^2 - \tilde{\cL}\abs{\nabla \mathcal{U}(t+s,s,x)}^2\notag\\
		\leq& -C_6\abs{D^2\mathcal{U}(t+s,s,x)}^2 +C_7 \abs{\nabla \mathcal{U}(t+s,s,x)}^2.
	\end{align}
	We choose $\gamma<\delta$ and multiply $e^{\gamma t}$ on both sides of the above inequality.  It follows that
	\begin{align*}
		&\int_0^T e^{\gamma t}\int_\mathbb{S}\frac{d}{dt}\abs{\nabla \mathcal{U}(t+s,s,x)}^2 d\apm dt -  \int_0^T e^{\gamma t}\int_\mathbb{S}\tilde{\cL}\abs{\nabla \mathcal{U}(t+s,s,x)}^2d\apm dt\\
		\leq&  -C_6 \int_0^T e^{\gamma t}\int_\mathbb{S}\abs{D^2\mathcal{U}(t+s,s,x)}^2d\apm dt   +C_7\int_0^Te^{\gamma t} \int_\mathbb{S}\abs{\nabla \mathcal{U}(t+s,s,x)}^2d\apm dt.
	\end{align*}
	Following Corollary \ref{corollary of L^*q(s,x)=0}, we see that
	$\int_0^T e^{\gamma t}\int_{\mathbb{S}}\tilde{\cL}\abs{\nabla \mathcal{U}(t+s,s,x)}^2d\apm dt=0.$
	Thus
	\begin{align*}
		\int_0^T e^{\gamma t}\int_\mathbb{S}\frac{d}{dt}\abs{\nabla \mathcal{U}(t+s,s,x)}^2 d\apm dt\leq C_7\int_0^Te^{\gamma t} \int_\mathbb{S}\abs{\nabla \mathcal{U}(t+s,s,x)}^2d\apm dt.
	\end{align*}
	Now by integration by parts, we note that
	{\small \begin{align*}
			&\int_0^T e^{\gamma t}\int_\mathbb{S}\frac{d}{dt}\abs{\nabla \mathcal{U}(t+s,s,x)}^2 d\apm dt\\
			=&e^{\gamma T}\int_\mathbb{S}\abs{\nabla \mathcal{U}(T+s,s,x)}^2 d\apm - \int_\mathbb{S}\abs{\nabla \mathcal{U}(s,s,x)}^2 d\apm-\gamma\int_0^Te^{\gamma t} \int_\mathbb{S}\abs{\nabla \mathcal{U}(t+s,s,x)}^2d\apm dt.
	\end{align*}}\\
	Then apply (\ref{inequality of boundedness of first derivative of u under apm and time int}) with $\gamma$ small enough and the boundedness of $\int_{\mathbb{S}}|\nabla \tilde{\phi}|^2 d\apm$ to have 
	\begin{align*}
		& e^{\gamma T}\int_\mathbb{S}\abs{\nabla \mathcal{U}(T+s,s,x)}^2 d\apm \\
		\leq&\int_\mathbb{S}\abs{\nabla \tilde{\phi}}^2 d\apm + (\gamma+C_7)\int_0^Te^{\gamma t} \int_\mathbb{S}\abs{\nabla \mathcal{U}(t+s,s,x)}^2d\apm dt\leq  C_8.
	\end{align*}
	Thus we obtained (\ref{exponential contraction of D^pU}) for the case when $p=1$.
	Now we continue to prove the induction step in the following content. Assume that for any $k\leq m$, there exist strictly positive constants $C_k$ and $\gamma_k$ such that for any $t>0$,
	$$\frac{1}{\tau}\int_0^{\tau}\int_{\R^d}\abs{D^k \mathcal{U}(t+s,s,x)}^2 q(s,x) dxds \leq C_k \exp(-\gamma_k t).$$
	Here we need to compare the expansion of the operators $\frac{d}{dt}$ and $\tilde{\cL}$ in the following:
	\begin{align*}
		\abs{D^m \mathcal{U}(t+s,s,x)}^2 = \sum_{\abs{J}=m}(\partial_J \mathcal{U}(t+s,s,x))^2,
	\end{align*}
	where $J$ is the multi-index with length $\abs{J}=m$. The multi-indices $J_a$ and $J_b$ are introduced for the following identity,
	\begin{align*}
		&\frac{d}{dt} \abs{D^m \mathcal{U}(t+s,s,x)}^2-\tilde{\cL}\abs{D^m \mathcal{U}(t+s,s,x)}^2\\
		=& -a_{ij}\left( \partial_{J}\partial_i\mathcal{U}(t+s,s,x)\right) \left( \partial_{J}\partial_j\mathcal{U}(t+s,s,x)\right)\\
		&+ \sum_{\abs{J_a}+\abs{J_b}\leq 2m+1}  \Phi_{J_a,J_b}^J \partial_{J_a} \mathcal{U}(t+s,s,x)\partial_{J_b} \mathcal{U}(t+s,s,x).
	\end{align*}
	Here the notation $\Phi_{J_a,J_b}^J$ contains all the combinations of spatial derivatives on the functions $a$ and $b$ with respect multi-indices $J_a$ and $J_b$ under some specified $J$. It is obvious the length of $J_a$ and $J_b$ will not exceed $m+1$. The boundedness of each elements in $\Phi_{J_a,J_b}^J$ comes from Condition (1). Therefore we will always have the following result by Young's inequality,
	\begin{align*}
		&\frac{d}{dt} \abs{D^m \mathcal{U}(t+s,s,x)}^2-\tilde{\cL}\abs{D^m \mathcal{U}(t+s,s,x)}^2\\
		\leq& -C_1^m \abs{D^{m+1} \mathcal{U}(t+s,s,x)}^2 +C_2^m \sum_{k\leq m} \abs{D^{k} \mathcal{U}(t+s,s,x)}^2 .
	\end{align*}
	Then we choose a strictly positive constant $\delta_{m+1}$ small enough to proceed as in (\ref{inequality of boundedness of first derivative of u under apm and time int}). Multiplying $e^{\delta_{m+1}t}$ on both sides and integrating with respect to $\apm$, we will have 
	$$
	\int_0^{\infty} e^{\delta_{m+1} t} \left( \int_{\R^d} \abs{D^{m+1}\mathcal{U}(t+s,s)}^2 d\apm \right) dt< \infty.
	$$
	Consider a higher order 
	\begin{align*}
		&\frac{d}{dt} \abs{D^{m+1} \mathcal{U}(t+s,s,x)}^2-\tilde{\cL}\abs{D^{m+1} \mathcal{U}(t+s,s,x)}^2\\
		\leq& -C_1^{m+1} \abs{D^{m+2} \mathcal{U}(t+s,s,x)}^2 +C_2^{m+1} \sum_{k\leq {m+1}} \abs{D^{k} \mathcal{U}(t+s,s,x)}^2 .
	\end{align*}
	By choosing $\gamma_{m+1}<\delta_{m+1}$ and following the same procedure as above, we have (\ref{exponential contraction of D^pU}) for the case when $p=m+1$.
	By induction principle, we proved the above result holds for any order of spatial derivatives of $\mathcal{U}(t+s,s,x)$. 
	
	
	By (\ref{equation of definition of density function by periodic measure}), we can conclude that $q(s,x)>0$ as $p(s+\tau,s,y,x)>0$ for any $s\in\R$ and $x,y\in\R^d.$ We can also prove the continuity of $q(t,x)$ from the continuity of $p(t+\tau,t,y,x)$ in $x$. Thus the density function $q(s,x)$ is strictly positive continuous function on any ball $B=B(0,R).$ It turns out that there exists $C>0$ such that
	$$\frac{1}{\tau}\int_0^{\tau}\norm{\partial_J \mathcal{U}(t+s,s)}_{L^2(B)}^2 ds \leq \frac{C}{\tau}\int_0^{\tau}\int_{\R^d}\abs{\partial_J \mathcal{U}(t+s,s,x)}^2 q(s,x)dxds.$$
	By the Sobolev embedding $W^{k,2}(B)\hookrightarrow C(B)$ for $k>\frac{d}{2}$ (\cite{Adams-Fournier2009}), we have that
	$$\frac{1}{\tau}\int_0^{\tau}\abs{\mathcal{U}(t+s,s,x)} ds \leq \frac{1}{\tau}\int_0^{\tau}\int_B\abs{D^k\mathcal{U}(t+s,s,x)}^2 dxds \leq C_{k}\exp (-\lambda_{k} t),$$
	for 
	any $x\in B$, where $k>\frac{d}{2}$. The proof is completed.
\end{proof}

\subsection{Estimates on the average of $\mathcal{U}(t+s,s)$ in $L^2(\pi_r)$}
In Section \ref{section 3.1}, we obtained the exponential contraction of $\frac{1}{\tau}\int_0^{\tau}\abs{\mathcal{U}(t+s,s,x)} ds$ in any ball $B$ when we assumed $\int_0^\tau\int_{\R^d}\phi(x)q(s,x)dxds=0.$ To consider the behaviour outside of the ball $B$, we need to introduce the weight $\pi_r(s,x)$ with some integer $r$ determined later,
$$\pi_r(s,x) = 1/{(2+\abs{x}^2+cos(\frac{2\pi s}{\tau}))^r}.$$
We consider its gradient and partial derivatives with respect to time $s$,
\begin{align*}
	\nabla \pi_r(s,x) = -\frac{2rx}{2+\abs{x}^2+cos(\frac{2\pi s}{\tau})}\pi_r(s,x),\ \frac{\partial}{\partial s} \pi_r(s,x) = \frac{\frac{2\pi r}{\tau} sin(\frac{2\pi s}{\tau})}{2+\abs{x}^2+cos(\frac{2\pi s}{\tau})} \pi_r(s,x).
\end{align*}

In general, it is easy to see that for any multi-index $J$ and any integer $r$, there exist smooth functions $\psi_{J,r}(s,x)$ and $\psi_{s,r}(s,x)$ such that,
\begin{align*}
	\partial_J\pi_r(s,x) &= \psi_{J,r}(s,x) \pi_r(s,x),\ \frac{\partial}{\partial s}\pi_r(s,x) = \psi_{s,r}(s,x) \pi_r(s,x),
\end{align*}
where $\psi_{J,r}(s,x)\rightarrow0$ and $\psi_{s,r}(s,x)\rightarrow0$ when $\abs{x}\rightarrow +\infty$.


\begin{lemma}\label{lemma of exponential contraction of u under weight pi_r(s,x)}
	Assume Conditions (1), (2) and (3), there exist strictly positive constants $C$ and $\lambda$ such that for any $t>0$, we have
	$$\frac{1}{\tau}\int_0^{\tau}\int_{\R^d}\abs{\mathcal{U}(t+s,s,x)}^2 \pi_r(s,x)dx ds \leq C\exp(-\lambda t).$$
\end{lemma}

\begin{proof}
	Recall (\ref{inequality of kunita's estimation}) to lead that
	for any integer $n\geq 0,$ it is possible to choose an integer $r_n$ such that, for any $0\leq m\leq n$, $t\geq 0$, we have
	$\abs{D^m \mathcal{U}(t+s,s,x)}\pi_{r_n}(s,x)\in L^2(\R^d).$
	We denote the multi-index $I$ for the derivative $\partial$ with the length $\abs{I}$. Consider the integer $M_I$ defined by
	$\abs{I}=[M_I-d/2],$
	and the property of the weight $\pi_{r}$, we have that there exists an integer $r_0$ such that for any $t>0$, any $r\geq r_0$ and any $m\leq M_I$,
	$D^m \left( \mathcal{U}(t+s,s,x)\pi_r(s,x)  \right)\in L^2(\R^d).$
	It is easy to see the periodicity of the function $\mathcal{U}(t+s,s,x)\pi_r(s,x)$ with respect to the initial time $s$. Note any order of its spatial derivatives are also $\tau$-periodic in $s$, so by integration by parts formula and periodicity,
	\begin{align*}
		&\int_0^{\tau}\int_{\R^d} \frac{d}{dt} \abs{\mathcal{U}(t+s,s)}^2 \pi_r dx ds\\
		=&-\int_0^{\tau}\int_{\R^d} (\partial_ib_i) \abs{\mathcal{U}(t+s,s)}^2 \pi_r dx ds -\int_0^{\tau}\int_{\R^d} b_i \abs{\mathcal{U}(t+s,s)}^2 (\partial_i\pi_r)dx ds \\
		&-\int_0^{\tau}\int_{\R^d}  \abs{\mathcal{U}(t+s,s)}^2 \left(\frac{\partial}{\partial s}\pi_r\right)dx ds\\
		&-\int_0^{\tau}\int_{\R^d} (\partial_i a_{ij})\mathcal{U}(t+s,s)(\partial_j \mathcal{U}(t+s,s))  \pi_r dx ds \\
		&-\int_0^{\tau}\int_{\R^d}  a_{ij}(\partial_i \mathcal{U}(t+s,s))(\partial_j \mathcal{U}(t+s,s)) \pi_r dx ds\\
		&-\int_0^{\tau}\int_{\R^d} a_{ij}\mathcal{U}(t+s,s)(\partial_{j} \mathcal{U}(t+s,s))  (\partial_i\pi_r)dx ds.
	\end{align*} 
	By Condition (2) and the property of the weight $\pi(s,x)$, we have that
	\small \begin{align*}
		&\int_0^{\tau}\int_{\R^d} \frac{d}{dt} \abs{\mathcal{U}(t+s,s)}^2 \pi_r dxds\\
		\leq& -\int_0^{\tau}\int_{\R^d} (\partial_ib_i) \abs{\mathcal{U}(t+s,s)}^2 \pi_rdxds  +\int_0^{\tau}\int_{\R^d} \frac{2r\cdot x\cdot b(s,x)}{2+\abs{x}^2+cos(\frac{2\pi s}{\tau})} \abs{\mathcal{U}(t+s,s)}^2 \pi_rdxds\\
		&-\int_0^{\tau}\int_{\R^d}  \abs{\mathcal{U}(t+s,s)}^2 \psi_s\pi_rdxds-\alpha \int_0^{\tau}\int_{\R^d}  \abs{\nabla \mathcal{U}(t+s,s)}^2 \pi_rdxds\\
		&+\frac{1}{2}\int_0^{\tau}\int_{\R^d} (\partial_{ij} a_{ij}) \abs{\mathcal{U}(t+s,s)}^2 \pi_rdxds + \frac{1}{2}\int_0^{\tau}\int_{\R^d} (\partial_{i} a_{ij}) \abs{\mathcal{U}(t+s,s)}^2 \psi_{j,r}\pi_rdxds\\
		&+\frac{1}{2}\int_0^{\tau}\int_{\R^d} (\partial_{j} a_{ij}) \abs{\mathcal{U}(t+s,s)}^2 \psi_{i,r}\pi_rdxds + \frac{1}{2}\int_0^{\tau}\int_{\R^d}  a_{ij} \abs{\mathcal{U}(t+s,s)}^2 \psi_{ij,r}\pi_rdxds\\
		=& \int_0^{\tau}\int_{\R^d}  \left( \Phi_{a,b}(s,x) + \Phi_{\psi}(s,x) +\frac{2r\cdot x\cdot b(s,x)}{2+\abs{x}^2+cos(\frac{2\pi s}{\tau})} \right)\abs{\mathcal{U}(t+s,s)}^2 \pi_rdxds \\
		&  -\alpha \int_0^{\tau}\int_{\R^d}  \abs{\nabla \mathcal{U}(t+s,s)}^2 \pi_rdxds,
	\end{align*}
	where $\Phi_{a,b}$ is a bounded function depending on functions $a$, $b$ and their derivatives, $\Phi_{\psi}$ is a function which could depend on functions $\psi_{I,r}$. It is easy to prove that $\Phi_{a,b}$ is independent of $r$. We also know that $\Phi_{\psi}$ tends to 0 when $\abs{x}$ goes to $\infty$. Therefore, we choose $r\geq r_0$ large enough to obtain,
	\begin{align}\label{inequality of derivatives outside of ball}
		\limsup_{\abs{x}\rightarrow \infty}\left( \Phi_{a,b}(s,x) + \Phi_{\psi}(s,x) +\frac{2r\cdot x\cdot b(s,x)}{2+\abs{x}^2+cos(\frac{2\pi s}{\tau})} \right) <0.
	\end{align}
	Now choosing the ball $B=B(0,R)$ with $R$ being large enough, which depends on the integer $r$, we have the following result from (\ref{inequality of derivatives outside of ball}),
	\begin{align*}
		&\int_0^{\tau}\int_{\R^d}  \left( \Phi_{a,b}(s,x) + \Phi_{\psi}(s,x) +\frac{2r\cdot x\cdot b(s,x)}{2+\abs{x}^2+cos(\frac{2\pi s}{\tau})} \right)\abs{\mathcal{U}(t+s,s)}^2 \pi_rdxds\\
		\leq & C_1 \int_0^{\tau}\int_{B} \abs{\mathcal{U}(t+s,s)}^2 \pi_rdxds -C_2  \int_0^{\tau}\int_{B^c} \abs{\mathcal{U}(t+s,s)}^2 \pi_rdxds \\
		\leq&( C_1+C_2)\int_0^{\tau}\int_{B} \abs{\mathcal{U}(t+s,s)}^2 \pi_rdxds - C_2\int_0^{\tau}\int_{\R^d} \abs{\mathcal{U}(t+s,s)}^2 \pi_rdxds,
	\end{align*}
	where $C_1,C_2>0$ are constants.
	Therefore, by Lemma \ref{lemma of exponential contraction of u for any Ball},
	\begin{align*}
		&\frac{d}{dt}\int_0^{\tau}\int_{\R^d}  \abs{\mathcal{U}(t+s,s)}^2 \pi_r dxds\leq - C_2\int_0^{\tau}\int_{\R^d} \abs{\mathcal{U}(t+s,s)}^2 \pi_rdxds+ C_3\exp (-\lambda t).
	\end{align*}
	The result follows then from the Gronwall's inequality.
\end{proof}

\subsection{Exponential decay of the spatial derivatives of the solution}

\begin{theorem}\label{theorem of exponential contraction of spacial derivatives for u}
	Assume Conditions (1), (2) and (3), and $\phi\in C_p^{\infty}$. Then 
	for any multi-index $I$, there exists an integer $k_I$, strictly positive constants $\Gamma_I$ and $\gamma_I$ 
	such that 
	$\frac{1}{\tau}\int_0^{\tau}\abs{\partial_I \mathcal{U}(t+s,s,x)} ds\leq \Gamma_I (1+\abs{x}^{k_I})\exp(-\gamma_I t).$
\end{theorem}
\begin{proof}
	The process of the proof is similar to Lemma \ref{lemma of exponential contraction of u for any Ball}. We first apply induction method on each order of spatial derivatives of $\mathcal{U}(t+s,s,x)$. 
	It guaranteed first the exponential contraction in any ball $B(0,R)$. Now we consider the behaviour outside of the ball to have,
	\begin{align*}
		&\int_{B^c}\tilde{\cL}\abs{\mathcal{U}(t+s,s)}^2\pi_r d\tilde{x}\\
		=&\int_{B^c}\left( \Phi_{a,b} + \Phi_{\psi} +\frac{2r\cdot x\cdot b}{2+\abs{x}^2+cos(\frac{2\pi s}{\tau})} \right)\abs{\mathcal{U}(t+s,s)}^2\pi_r d\tilde{x}<0,
	\end{align*}
	if we choose the ball large enough. 
	Thus we have some positive $C_0$ and $\lambda_0$ such that 
	\begin{align}\label{inequality of exponential contraction for L on u}
		\int_\mathbb{S}\tilde{\cL}\abs{\mathcal{U}(t+s,s)}^2\pi_r d\tilde{x}<\int_{B}\tilde{\cL}\abs{\mathcal{U}(t+s,s)}^2\pi_r d\tilde{x}\leq C_0 \exp(-\lambda_0 t).
	\end{align}
	On the other hand, we integrate with respect to $d\tilde{x}$ with weight $\pi_r,$
	multiply $e^{\delta t}$ and integrate with respect to $t$ from 0 to $T$ to have
	{\small \begin{align*}
			&e^{\delta T}\int_\mathbb{S}\abs{\mathcal{U}(T+s,s)}^2\pi_r d\tilde{x}+C\int_0^T e^{\delta t}\int_\mathbb{S}\abs{\nabla \mathcal{U}(t+s,s)}^2\pi_r d\tilde{x}dt\\
			\leq &\int_\mathbb{S}\abs{\tilde{\phi}(\tilde{x})}^2\pi_r d\tilde{x} + \delta\int_0^Te^{\delta t} \int_\mathbb{S}\abs{\mathcal{U}(t+s,s)}^2\pi_r d\tilde{x} dt + \int_0^Te^{\delta t} \int_\mathbb{S}\tilde{\cL}\abs{\mathcal{U}(t+s,s)}^2\pi_r d\tilde{x} dt.
	\end{align*}}
	By the estimates (\ref{inequality of exponential contraction for general t}) and (\ref{inequality of exponential contraction for L on u}), we can choose constant $\delta$ small enough to obtain
	\begin{align*}
		\int_0^T e^{\delta t}\int_\mathbb{S}\abs{\nabla \mathcal{U}(t+s,s)}^2\pi_r d\tilde{x}dt\leq C.
	\end{align*}
	Similarly we consider (\ref{inequality of comparison between L and dt with C}) to have
	\begin{align*}
		&e^{\gamma T}\int_\mathbb{S}\abs{\nabla \mathcal{U}(T+s,s)}^2\pi_r d\tilde{x}+C\int_0^T e^{\gamma t}\int_\mathbb{S}\abs{D^2 \mathcal{U}(t+s,s)}^2\pi_r d\tilde{x}dt\\
		\leq &\int_\mathbb{S}\abs{\nabla \tilde{\phi}(\tilde{x})}^2\pi_r d\tilde{x} + (\gamma+C_2)\int_0^Te^{\gamma t} \int_\mathbb{S}\abs{\nabla \mathcal{U}(t+s,s)}^2\pi_r d\tilde{x} dt \\
		&+ \int_0^Te^{\gamma t} \int_\mathbb{S}\tilde{\cL}\abs{\nabla \mathcal{U}(t+s,s)}^2\pi_r d\tilde{x} dt,
	\end{align*}
	which gives us the conclusion that
	$\int_{\mathbb{S}}\abs{\nabla \mathcal{U}(t+s,s)}^2\pi_r d\tilde{x}\leq C e^{-\gamma t}.$
	It is easy to repeat the process for any $m\in\N$ with positive constants $C_m$ and $\gamma_m$ to obtain 
	$\int_{\mathbb{S}}\abs{D^m \mathcal{U}(t+s,s)}^2\pi_r d\tilde{x}\leq C_m e^{-\gamma_m t}.$
	Then we proved the conclusion of the theorem by the weighted Sobolev embedding Theorem with $\pi_r(s,x)d\tilde{x}$ instead of the the density function of average periodic measure $q(s,x)d\tilde{x}$ .
\end{proof}
The following remark applies to general case without assumption (\ref{equality of annihilation of initial condition}).
\begin{remark}
	The proof of the previous theorem also gives us the result that there exist some integer $l\in\N$ and constants $\Gamma>0,$ $\gamma>0$, such that for any $t$ and $x$, 
	\begin{align}\label{inequality of convergence from initial distribution to periodic measure}
		\abs{\frac{1}{\tau}\int_0^{\tau} \mathcal{U}(t+s,s,x)ds  -  \int_\mathbb{S}\tilde{\phi}(\tilde{x})d\apm(\tilde{x})} \leq \Gamma(1+\abs{x}^l)\exp(-\gamma t).
	\end{align}
	
\end{remark}

\section{Ergodicity for discretized semi-flows of Euler-Maruyama scheme}
We consider Euler-Maruyama numerical scheme with step size $\dt=\frac{\tau}{N}>0$ for SDE (\ref{equations of SDE}):
\begin{eqnarray}\label{equation of scheme of discrete random periodic solution by steps for weak approximation}
	\hxit{(i+1)}  =\hxit{i}+b(i\dt,\hxit{i})\dt+\sigma(\hxit{i})\dW_i,
\end{eqnarray}
with $\hx_{-k\tau}^{-k\tau}=x$, where $i=0,1,2,\ldots$, $\dW_i=W_{-k\tau+(i+1)\dt}-W_{-k\tau+i\dt}$.
There are several methods to generate the stochastic increment $\dW_i$. But in order to obtain the ergodicity of numerical schemes, in this paper we apply Gaussian distribution in  the approximation i.e. $\dW=\sqrt{\dt}\mathcal{N}(0,1)$.
Denote transition probability
$$\hat P(-k\tau+i\dt,-k\tau,x,\Gamma)=\mathbb{P}\left(\hxit{i}\in\Gamma\right).$$
It is easy to see that 
$\hat P(-k\tau+i\dt,-k\tau,x,\Gamma)=\hat P(i\dt,0,x,\Gamma).$ One can easily extend the numerical scheme (\ref{equation of scheme of discrete random periodic solution by steps for weak approximation}) to $\hx_{i\dt}^{j\dt}$, $i\geq j$, $j\in\Z$ with $\hx_{j\dt}^{j\dt}=x$ and its transition probability to $\hat P(i\dt,j\dt,x,\cdot)$, $i\geq j$, $j\in\Z.$ The corresponding semigroup $\hat{\mathcal{T}}(i\dt,j\dt)$, $i\geq j$, $j\in\Z$ can be generated from the transition probability in a standard way. A measure-valued function $\hat \rho:\Z\to\mathbb{P}(\R^d)$ is called a periodic measure of the semigroup $\hat{\mathcal{T}}(i\dt,j\dt)$ if 
$$\int_{\R^d}\hat P(i\dt,j\dt,x,\Gamma)\hat \rho_j(dx)=\hat \rho_i(\Gamma)$$ and
$$\hat \rho_{i+N}=\hat \rho_i$$
for all $i\in\Z.$ Recall here $N=\tau/\dt.$
\begin{remark}
	By Condition (1), if the function $b(t,0)$ is bounded for any $t>0$, then $b$ is of linear growth $\abs{b(t,X_t)}\leq L\abs{X_t}+C$, where $L,C>0$.
\end{remark}
\begin{remark}
	Under Condition (3), the conclusion in the following proposition still holds for sufficient small step size $\dt<\dt_c$ with some $\dt_c>0$ that may depend on the growth order of the test function. In order to obtain a uniform $\dt_c$, we consider the following slightly stronger condition. But in the case of one-dimension, Condition (3') is the same as Condition (3). This means that in the case of one-dimension and Condition (3), a uniform $\dt<\dt_c$ is obtained with respect to all polynomial growth test functions.
\end{remark}
{\bf Condition (3')}
{\it For all $i=1,2,\ldots,d$, there exist constants $\beta_i>0$ and $C_{\beta_i}>0$ such that for any $t\in\R^+$ and any $x_i\in\R$,
	$x_i\cdot b_i(t,x)\leq -\beta_i\abs{x_i}^2+C_{\beta_i}.$}
\begin{proposition}\label{proposition of bound of n-th moments of discrete process for weak approximation}
	Assume Conditions (1), (3') and the boundedness of $b(t,0)$ for any $t>0$, then for any integer $p$ and any $\delta>0$, there exist constants $C_p,\hat{C}_p,\gamma,\hat\gamma>0$ such that for any $0<\dt\leq\frac{2}{L+\delta}$, $x\in\R^d$ and $n\in\N$, 
	$$
	\E\abs{\hxit{n}}^p\leq C_p\rbrac{1+\abs{x}^p \exp(-\gamma p n\dt)},
	$$
	and
	\begin{align}\label{inequality of lyapunov condition}
		\E\sbrac{\abs{\hxit{(i+1)}}^p\bigg\vert\hat{\mathcal{F}}_i}\leq (1-\hat\gamma\dt)\abs{\hxit{i}}^p+\hat C_p,
	\end{align}
	where $\hat{\mathcal{F}}_i=\mathcal{F}_{-k\tau+i\dt}$.
\end{proposition}
\begin{proof}
	We first consider the one-dimensional case. Condition (3'), which is the same as Condition (3) in this case,  implies that for any $\abs{x}>\sqrt{\frac{C_\beta}{\beta}}$, $x\cdot b(t,x)\leq-\beta \abs{x}^2 + C_\beta< 0.$
	It then follows that when $x>\sqrt{\frac{C_\beta}{\beta}}$, 
	$$(1-L\dt)x-C\dt\leq x+b(t,x)\dt\leq (1-\beta\dt)x+\frac{C_\beta\dt}{x}\leq (1-\beta\dt)x+\sqrt{\beta C_\beta}\dt.$$
	Thus,
	\begin{eqnarray}\label{20204a}
		\abs{x+b(t,x)\dt}\leq\max\{\abs{1-\beta\dt},\abs{1-L\dt}\}\abs{x}+C_1\dt,
	\end{eqnarray}
	where $C_1$ is independent of $\dt$. One can obtain the same result for $x<-\sqrt{\frac{C_\beta}{\beta}}.$
	It is not hard to verify that $L\geq \beta$. Then, for ${\hat\gamma}_1=\min\{\beta,\delta\}>0$, we can see that for $0<\dt\leq\frac{2}{L+\delta}$,  
	\begin{eqnarray}\label{20204b}
		\max\{\abs{1-\beta\dt},\abs{1-L\dt}\}<1-{\hat\gamma}_1\dt.
	\end{eqnarray}
	It then follows from (\ref{20204a}) and (\ref{20204b}) that 
	\begin{eqnarray}\label{20204c}
		\abs{x+b(t,x)\dt}\leq\abs{1-{\hat\gamma}_1\dt}\abs{x}+C_1\dt.
	\end{eqnarray}
	Fix any sufficiently small $\hat \varepsilon>0$, choose $\varepsilon_k$ such that $\frac{p-k}{p}\varepsilon_k^{\frac{p}{p-k}}=\hat\varepsilon^k$. Now 
	for any given integer $p$, by Young's inequality with positive $\varepsilon_k$, which will be fixed later 
	\begin{eqnarray}\label{20204d}
		\abs{1-{\hat\gamma}_1\dt}^{p-k}\abs{x}^{p-k}C_1^k
		\leq
		\abs{1-{\hat\gamma}_1\dt}^{p}\abs{x}^{p}\left(\frac{p-k}{p}\varepsilon_k^{\frac{p}{p-k}}\right) +{(k C_1^p)}/{(p\varepsilon_k^{\frac{p}{k}})},
	\end{eqnarray}
	it then follows from (\ref{20204c}) and (\ref{20204d}) that 
	\begin{align*}
		&\abs{x+b(t,x)\dt}^p\\
		\leq&
		\abs{1-{\hat\gamma}_1\dt}^p\abs{x}^p+\sum_{k=1}^{p-1}\binom{p}{k}\abs{1-{\hat\gamma}_1\dt}^{p-k}\abs{x}^{p-k}C_1^k(\dt)^k+(C_1\dt)^p\\
		\leq&\left(1+\sum_{k=1}^{p-1}\binom{p}{k}(\dt)^k \hat\varepsilon^k \right)\abs{1-{\hat\gamma}_1\dt}^p\abs{x}^p+\sum_{k=1}^{p-1}\binom{p}{k}(\dt)^k\frac{k C_1^p}{p\varepsilon_k^{\frac{p}{k}}}  +(C_1\dt)^p.
	\end{align*}
	Now we choose $\hat\varepsilon$ small enough to obtain 
	$$\left(1+\sum_{k=1}^{p-1}\binom{p}{k}(\dt)^k\hat\varepsilon^k\right)\abs{1-{\hat\gamma}_1\dt}^p<(1+\hat\varepsilon\dt)^p\abs{1-{\hat\gamma}_1\dt}^p\leq\abs{1-{\hat\gamma}_2\dt}^p<1,$$
	with some constant ${\hat\gamma}_2>0$ being independent of $\dt$. 
	Then for any fixed $p$, 
	\begin{align*}
		\abs{x+b(t,x)\dt}^p
		\leq
		\abs{1-{\hat\gamma}_2\dt}^p\abs{x}^p+C_2\dt,
	\end{align*}
	with some constant $C_2$ independent of $\dt.$
	Denote by $A=\left\{x:\abs{x}\leq\sqrt{\frac{C_\beta}{\beta}}\right\}$. If the conclusion of this proposition holds for any even $p$, one can obtain the result for odd $p$ by
	\begin{align*}
		\E\abs{\hxit{n}}^p\leq &\sqrt{\E\abs{\hxit{n}}^{2p}}\leq \sqrt{C_{2p}\rbrac{1+\abs{x}^{2p} \exp(-2\gamma p n\dt)}}\\
		\leq& \sqrt{C_{2p}}\rbrac{1+\abs{x}^{p} \exp(-\gamma p n\dt)},
	\end{align*}
	with coefficient $C_p=\sqrt{C_{2p}}.$
	Then we only consider the cases where $p$ is even in the following. For this we apply the same argument using Young's inequality as above on (\ref{equation of scheme of discrete random periodic solution by steps for weak approximation}) and conditional expectation to have
	\begin{align*}
		&\E\sbrac{\abs{\hxit{(i+1)}}^pI_{A^c}\rbrac{\hxit{i}} \bigg\vert \hat{\mathcal{F}}_i}\\
		=&I_{A^c}\rbrac{\hxit{i}} \left[\rbrac{\hxit{i}+b(i\dt,\hxit{i})\dt}^p\right.\\
		&\left .+\sum_{l=1}^{p-1}\binom{p}{l}\rbrac{\hxit{i}+b(i\dt,\hxit{i})\dt}^{p-l}\rbrac{\sigma(\hxit{i})}^l\E\sbrac{\rbrac{\dW_i}^{l}\bigg\vert \hat{\mathcal{F}}_i}\right .\\
		&\left.+\rbrac{\sigma(\hxit{i})}^p\E\sbrac{\rbrac{\dW_i}^{p}\bigg\vert \hat{\mathcal{F}}_i}\right ]\\
		\leq&\left[\left(1+\hat \varepsilon\dt\right)^{\frac{p}{2}} \abs{1-\hat\gamma_{2}\dt}^p\abs{\hxit{i}}^{p}+C_3\dt\right]I_{A^c}\rbrac{\hxit{i}}\\
		\leq
		&
		\rbrac{\abs{1-\hat\gamma_{3}\dt}^{\frac{p}{2}}
			\abs{\hxit{i}}^{p}+C_3\dt} I_{A^c}\rbrac{\hxit{i}},
	\end{align*}
	with some constant $C_3$ independent of $\dt$, where $\hat\varepsilon$ is chosen small enough such that $(1+\hat\varepsilon\dt)^{\frac{p}{2}} \abs{1-\hat\gamma_{2}\dt}^p\leq \abs{1-\hat\gamma_{3}\dt}^{\frac{p}{2}}
	<1$ for some $\hat \gamma _3>0$. Moreover by linear growth condition of $b$ and $\abs{1+L\dt}\leq 1+L\cdot\frac{2}{L}=3$, we can obtain
	\begin{align*}
		\E\sbrac{\abs{\hxit{(i+1)}}^p I_{A}\rbrac{\hxit{i}}\bigg\vert \hat{\mathcal{F}}_i}
		\leq
		\rbrac{3^p\rbrac{\frac{C_\beta}{\beta}}^{\frac{p}{2}}+C} I_{A}\rbrac{\hxit{i}},
	\end{align*}
	where $L$ is the bound of function $b$'s first derivative (or coefficient of global Lipschitz).
	We combine the above two estimates to obtain
	\begin{align*}
		\E\sbrac{\abs{\hxit{(i+1)}}^p\bigg\vert\hat{\mathcal{F}}_i}
		\leq&\abs{1-\hat\gamma_{3}\dt}^{\frac{p}{2}}\abs{\hxit{i}}^{p}+\hat C_{p},
	\end{align*}
	with $\hat C_{p}=\max\left\{C_{3}\dt,3^p\rbrac{\frac{C_\beta}{\beta}}^{\frac{p}{2}}+C\right\}$ and $\hat\gamma=\hat\gamma_3$ 
	Therefore,
	\begin{align*}
		&\abs{1-\hat\gamma_{3}\dt}^{-{\frac{p}{2}}n}\E\sbrac{\abs{\hxit{n}}^p}\\
		=&\abs{x}^p+\sum_{i=1}^{n}\left\{\abs{1-\hat\gamma_{3}\dt}^{-{\frac{p}{2}}i}\E{\abs{\hxit{i}}^p}-\abs{1-\hat\gamma_{3}\dt}^{-{\frac{p}{2}}(i-1)}\E{\abs{\hxit{(i-1)}}^p}\right\}\\
		=&\abs{x}^p+\sum_{i=1}^{n}\abs{1-\hat\gamma_{3}\dt}^{-{\frac{p}{2}}i}\E\sbrac{\E\left(\abs{\hxit{i}}^p-\abs{1-\hat\gamma_{3}\dt}^{\frac{p}{2}}\abs{\hxit{(i-1)}}^p\bigg\vert \hat{\mathcal{F}}_{i-1}\right)}\\
		\leq& \abs{x}^p +  \hat C_{p}\sum_{i=1}^{n}  \abs{1-\hat\gamma_{3}\dt}^{-{\frac{p}{2}}i}
		\leq \abs{x}^p + C_{p},
	\end{align*}
	where $C_p=\frac{\hat C_{p}}{1-\abs{1-\hat\gamma_{3}\dt}^{\frac{p}{2}}} .$ Finally from $\abs{1-\hat\gamma_{3}\dt}<\exp(-\hat \gamma_3\dt)$, 
	we have that $\E\sbrac{\abs{\hxit{n}}^p}\leq\abs{x}^p\exp(-{\frac{\hat \gamma_3}{2}}p n\dt) +  C_p.$

	For the multi-dimensional case, we apply Condition (3') to have the estimations with coefficients $C_{i,p}$ and $\gamma_{i}$ for each $i=1,2,\ldots,d$. Then the final conclusion follows.
\end{proof}

\begin{proposition}\label{propostion of numerical scheme is ergodic}
	Assume the conditions in Proposition \ref{proposition of bound of n-th moments of discrete process for weak approximation} and Condition (2), then the Euler-Maruyama scheme (\ref{equation of scheme of discrete random periodic solution by steps for weak approximation}) is geometrically ergodic for all step-size $0<\dt<\frac{2}{L}$, i.e. there exists a periodic measure $\hat\rho^\dt:\Z\to\mathcal{P}(\R^d)$ such that
	$$\norm{\hat P(i\dt,-k\tau+i\dt,x,\cdot)-\hat\rho_i(\cdot)}_{TV}\leq Ce^{-\delta k\tau},\quad k\in\N,$$
	for some constants $C,\delta>0.$ 
\end{proposition}
\begin{proof}
	To check the local Doeblin condition (\cite{Meyn-Tweedie1993},
\cite{Nummelin1984}), we need to prove the transition kernel
	$\hat{P}(s+\tau,s,x,A)=\hat{P}(s+N\dt,s,x,A)=\hat{\mathbb{P}}\{\hx_{s+N\dt}^s\in A\vert \hx_s^s=x\}$
	possesses a density function $\hat{p}(t,s,x,y)$ satisfying
	$\inf_{x,y\in K}\hat{p}(s+\tau,s,x,y)>0$
	for some non-empty set $K\in\mathcal{B}$ with Lebesgue measure $\Lambda(K)>0.$
	To apply the result of Theorem 3.5 in \cite{Feng-Zhao-Zhong2019measure} to prove the local Doeblin condition of the transition kernel,
	we only need to prove there is a non-empty compact set $K$, such that for any $i=0,\ldots,N-1$ and any non-empty open set $\Gamma$, we have
	\begin{align}\label{local doeblin for one step}
		\inf_{x\in K}\hat{P}(s+(i+1)\dt,s+i\dt,x,\Gamma)>0.
	\end{align}

	Consider the numerical approximation in the time interval $[s+i\dt,s+(i+1)\dt]$. For simplicity, we denote by $s_i=s+i\dt,$ $i=0,\ldots,N-1$ and $\hx_t^i=\hx_{s_i+t}^{s,x}$. So
	\begin{align}\label{equation 4.8}
		\hx_\dt^i-\hx_0^i=b(s_i,\hx_0^i)\dt+\sigma(\hx_0^i)\dW_t,
	\end{align}
	where $\dW_t=W_{\dt}-W_{0}$. Let $\hx_{\dt}^{i,x}$ be defined by (\ref{equation 4.8}) conditioned on $\hx_0^i=x$. Then the law of $\hx_{\dt}^{i,x}$ is $\hat{P}(s+(i+1)\dt,s+i\dt,x,\cdot)=\hat{P}(s_i+\dt,s_i,x,\cdot)$. Note $b(s_i,x)$ and $\sigma(x)$ are non-random and given, thus $\hat{P}(s_i+\dt,s_i,x,\cdot)$ is simply the Gaussian distribution with mean $x+b(s_i,x)\dt\in\R^d$ and covariance matrix $\sigma\sigma^T(x)\dt\in\R^{d\times d}$. The covariance matrix is uniformly non-degenerate, thus for any non-empty open set $\Gamma\in\R^d,$ we have 
	$\hat{P}(s_i+\dt,s_i,x,\Gamma)>0$
	and the function $x\mapsto \hat{P}(s_i+\dt,s_i,x,\Gamma)$ is continuous. Thus for any compact set $K\subset\R^d,$ we have (\ref{local doeblin for one step}).

	By Theorem 3.5 in \cite{Feng-Zhao-Zhong2019measure}, we obtain the local Doeblin condition of 
	$\hat{P}(s+N\dt,s,x,\cdot)$. Note condition $0<\dt<\frac{2}{L}$ implies there exists $\delta>0$ such that $0<\dt<{\frac{2}{L+\delta}}$. So 
	Proposition \ref{proposition of bound of n-th moments of discrete process for weak approximation}  holds and estimate (\ref{inequality of lyapunov condition}) implies Lyapunov condition with Lyapunov function $V(x)=x^2$. Then by Theorem 3.3 in \cite{Feng-Zhao-Zhong2019measure}, we deduce the ergodicity of the numerical scheme and the convergence to the periodic measure $\hat\rho^{\dt}.$
\end{proof}
Similar in the continuous time case, we can lift the discrete semi-flow and periodic measure to $\left\{0,1,2,\ldots,N\right\}\times\R^d$ as follows
$$\widetilde{\hx}_{i\dt}^{j\dt}=(i,{\hx}_{i\dt}^{j\dt}),\quad i\geq j,$$
$$\tilde{\hat P}(i,(j,x),\{k\}\times\Gamma)=\delta_{(i+j\ mod\ N)}(k)\hat P((i+j)\dt,j\dt,x,\Gamma),\quad i\geq j,$$
$$\tilde{\hat \rho}_i^\dt(\{k\}\times\Gamma)=\delta_{(i\ mod\ N)}(k)\hat\rho_i^\dt(\Gamma),\quad i\in\Z.$$
Then $\widetilde{\hx}$ is a cocycle and $\tilde{\hat P}$ is the transition probability of $\widetilde{\hx}$. Moreover, $\tilde{\hat \rho}^\dt$ is the periodic measure of $\tilde{\hat P}$, i.e. 
$$\sum_{l=0}^{N-1}\int_{{\R^d}}\tilde{\hat P}(i,(l,x),\{k\}\times\Gamma)\tilde{\hat \rho}_j^\dt(\{l\}\times dx)=\tilde{\hat \rho}_{i+j}^\dt(\{k\}\times\Gamma).$$
Define $\bar{\tilde{\hat \rho}}^\dt=\frac{1}{N}\sum_{j=0}^{N-1}\tilde{\hat \rho}_j^\dt.$ Then it is easy to see that 
$$\sum_{l=0}^{N-1}\int_{{\R^d}}\tilde{\hat P}(i,(l,x),\{k\}\times\Gamma)\bar{\tilde{\hat \rho}}^\dt(\{l\}\times dx)=\bar{\tilde{\hat \rho}}^\dt(\{k\}\times\Gamma),$$
i.e. $\bar{\tilde{\hat \rho}}$ is the invariant measure of $\tilde{\hat P}(i)$, $i\in\N.$ Moreover, for any measurable function $\phi:\R^d\to\R,$
\begin{align}\label{equivalence of measurable function under lifted and original discrete apm}
	\sum_{k=0}^{N-1}\int_{\R^d}\phi(x)\bar{\tilde{\hat \rho}}^\dt(\{k\}\times dx) =& \sum_{k=0}^{N-1}\int_{\R^d}\phi(x)\frac{1}{N}\sum_{j=0}^{N-1}\delta_{j}(k){\hat \rho}_j^\dt(dx)\\
	=&\frac{1}{N}\sum_{k=0}^{N-1}\int_{\R^d}\phi(x){\hat \rho}_k^\dt(dx)\notag\\
	=&\sum_{k=0}^{N-1}\int_{\R^d}\phi(x)\bar{\hat \rho}^\dt(dx),\notag
\end{align}
where $\bar{\hat \rho}^\dt=\frac{1}{N}\sum_{j=0}^{N-1}{\hat \rho}_j^\dt.$

\section{Error estimate for the approximation to periodic measures}
In autonomous systems, there are some established results in the ergodicity of numerical schemes (Mattingly, Stuart and Higham \cite{mattingly-stuart-higham}; Grorud and Talay \cite{Grorud-Talay1996}; Talay \cite{Talay1990}, \cite{Talay1991}). But in our non-autonomous model, due to the lacking of weakly mixing property, those approaches may not give immediately the error between $\int_{\R^d}{\phi}({x})\bar{\rho}(d{x})$ and $\int_{\R^d}{\phi}({x})\bar{\hat \rho}^\dt(d{x})$ over $[0,\tau]$. We develop the following approach using integration with respect to initial time $s$ to obtain the error estimate of invariant measure. To approximate the average of periodic measure, we need to consider the long time behaviour of the SDE (\ref{equations of SDE}) by pullback of the initial time to $s-k\tau$. First we notice from the ergodic theory and ergodicity of $\{X_t^{s,x}\}_{t\geq s}$ and $\{\hx_{s+i\dt}^{s,x}\}_{i\geq 0}$, we have the following law of large numbers. Recall (\ref{equivalence of measurable function under lifted and original apm}) and (\ref{equivalence of measurable function under lifted and original discrete apm}), so for $\phi$ with at most polynomial growth at infinity and as both of $X_t^{s,x}$ and $\hx_{s+i\dt}^{s,x}$ possess finite moments of any order, we have that for any $\varepsilon>0,$ there exists a constant $N'$ such that for all $n\geq N'$,
\begin{eqnarray}\label{law of large numbers for exact solution}
	\abs{\frac{1}{n}\sum_{k=1}^{n}\frac{1}{\tau}\int_0^{\tau}\E\phi(X_{s+k\tau}^{s,x})ds - \int_{\R^d}\phi(x)\bar{\rho}(dx)}&\leq \varepsilon,
	\quad a.s.,\\
	\label{law of large numbers for numerical approximation}
	\abs{\frac{1}{n}\sum_{k=1}^{n}\frac{1}{\tau}\int_0^{\tau}\E\phi(\hx_{s+kN\dt}^{s,x})ds - \int_{\R^d}\phi(x)\bar{\hat \rho}^\dt(dx)}&\leq \varepsilon,   \quad a.s..
\end{eqnarray}
\begin{theorem}\label{theorem of main result}
	Assume conditions in Proposition \ref{propostion of numerical scheme is ergodic}. Then for any step size $\dt=\tau/N$, $N\in\N$, satisfying $\dt<\frac{2}{L}$ and any function $\phi\in\mathcal{C}_p^{\infty},$ we have:
	\begin{align}\label{error between average of periodic measure and its numerical approximation}
		\abs{\int_{\R^d}{\phi}({x})\bar{\rho}(d{x})-\int_{\R^d}{\phi}({x})\bar{\hat \rho}^\dt(d{x})}= \mathcal{O}(\dt).
	\end{align}
\end{theorem}
\begin{proof}
	Define $\mathcal{U}(s+i\dt,s,x)=\E\phi(X_{s+i\dt}^{s,x})$. Then
	\begin{align}\label{equation of exact solution and numerical approximation}
		\mathcal{U}(s,s-k\tau,x)=\E\phi(X_{s}^{s-k\tau,x}),\quad \mathcal{U}(s,s,\hx_{s}^{s-k\tau,x}) = \phi(\hx_{s}^{s-k\tau,x}),\ a.s..
	\end{align}
	By the periodicity of $\mathcal{U}(t,s,x)$ with respect to initial time $s$, i.e. $\mathcal{U}(t,s-k\tau,x) = \mathcal{U}(t+k\tau,s,x)$, it is always possible to move the initial time into $[0,\tau)$.
	Now we consider the following It\^o-Taylor expansion:
	\begin{align}\label{equality of weak appro u_i}
		\E\mathcal{U}(s+k\tau,s+(i+1)\dt,\hx_{(i+1)\dt}^{s,x}) 
		=&   \E\mathcal{U}(s+k\tau,s+(i+1)\dt,\hx_{i\dt}^{s,x}) \notag \\
		&\hspace*{-4cm}+  \tilde{\cL}(s+(i+1)\dt)\E \mathcal{U}(s+k\tau,s+(i+1)\dt,\hx_{i\dt}^{s,x}) \dt +R_{1,i}^s(\dt)^2.
	\end{align}
	Denote $s'=s+(i+1)\dt$ and $t'=k\tau-(i+1)\dt$. Then it is obvious that
	\begin{align*}
		\E \mathcal{U}(s+k\tau,s+i\dt,x)=\E \mathcal{U}((s'-\dt)+(t'+\dt),s'-\dt,x).
	\end{align*}
	Therefore, we have the following It\^o-Taylor expansion:
	\begin{align}\label{equality of weak appro u_i+1}
		\E\mathcal{U}(s+k\tau,s+i\dt,\hx_{i\dt}^{s,x}) =& \E \mathcal{U}(s+k\tau,s+(i+1)\dt,\hx_{i\dt}^{s,x})   \notag \\
		&\hspace*{-2.5cm}+   \rbrac{\frac{\partial}{\partial t} - \frac{\partial}{\partial s}}\E \mathcal{U}(s+t,s,\hx_{i\dt}^{s,x})\dt\bigg\vert_{s=s',t=t'}    +R_{2,i}^s(\dt)^2.
	\end{align}
	The coefficients $R_{1,i}^s$ and $R_{2,i}^s$ have the following form:
	\begin{align}\label{ito-taylor expansions of u}
		\E\sbrac{ \psi(\hx_{i\dt}^{s,x})\cdot \partial_J \mathcal{U}\rbrac{s+k\tau,s+i\dt,\hx_{i\dt}^{s,x}+  \vartheta\rbrac{\hx_{(i+1)\dt}^{s,x}-\hx_{i\dt}^{s,x}}} },
	\end{align}
	where $0<\vartheta<1$ and the function $\psi(x)$ is a product of functions $b$, $\sigma$ and their derivatives. 
	One can obtain the boundedness of $\psi(x)$ from Condition (1). 
	Combining (\ref{equality of weak appro u_i}) and (\ref{equality of weak appro u_i+1}), we have
	\begin{align}
		&\E\mathcal{U}(s+k\tau,s+(i+1)\dt,\hx_{(i+1)\dt}^{s,x})-\E\mathcal{U}(s+k\tau,s+i\dt,\hx_{i\dt}^{s,x})\notag\\
		=&\rbrac{ \frac{\partial}{\partial s} - \frac{\partial}{\partial t}+ {\cL}(s)}\E \mathcal{U}(s+t,s,x)\dt\bigg\vert_{s=s',t=t',x=\hx_{i+1}^{s,x}} +(R_{2,i}^s-R_{1,i}^s)(\dt)^2.
	\end{align}
	As $\frac{\partial}{\partial s}+ {\cL}(s)=\frac{\partial}{\partial t},$ we take summation on both sides of the above and from (\ref{equation of exact solution and numerical approximation}), periodicity of $\mathcal{U}$,
	{\small \begin{align*}
			&\E\phi(\hx_{s}^{s-k\tau,x}) - \E\phi(X_{s}^{s-k\tau,x})\\
			=&\sum_{i=0}^{kN-1}\rbrac{\E\mathcal{U}(s+k\tau,s+(i+1)\dt,\hx_{(i+1)\dt}^{s,x})-\E\mathcal{U}(s+k\tau,s+i\dt,\hx_{i\dt}^{s,x})}=\sum_{i=0}^{kN-1}R_{i}^s(\dt)^2,
	\end{align*}}\\
	where $R_i^s = R_{1,i}^s - R_{2,i}^s.$
	Combining this with Proposition \ref{proposition of bound of n-th moments of discrete process for weak approximation} and Theorem \ref{theorem of exponential contraction of spacial derivatives for u}, there exists a constant $\lambda>0$ and an integer $l\in \N$, such that 
	\begin{align*}
		&\abs{\sum_{i=0}^{kN-1}\frac{1}{\tau}\int_0^\tau R_{i}^s ds}\\
		\leq& \sum_{i=0}^{kN-1} C\E \sbrac{  \frac{1}{\tau}\int_0^{\tau} \abs{\partial_J \mathcal{U}\rbrac{s+k\tau,s+i\dt,\hx_{i\dt}^{s,x}+  \vartheta\rbrac{\hx_{(i+1)\dt}^{s,x}-\hx_{i\dt}^{s,x}}}}ds }\\
		\leq&   \frac{C}{\tau}\sup_{i\geq 0} \E\rbrac{1  +  \abs{\hx_{-k\tau+i\dt}^{-k\tau,x}}^l   + \abs{\hx_{-k\tau+(i+1)\dt}^{-k\tau,x}}^l }\sum_{i=0}^{kN-1}\exp\rbrac{-\lambda(kN-i)\dt}  \\
		\leq& \frac{1-e^{-\lambda k N\dt}}{1-e^{-\lambda\dt}}e^{-\lambda\dt}C_l\rbrac{1+\abs{x}^l } .
	\end{align*}
	Let $k$ go to infinity and $\dt$ be small enough, we have
	\begin{align}\label{inequality of cumulation of error}
		\sum_{i=0}^{+\infty}\frac{1}{\tau}\int_0^\tau \abs{R_i^s}(\dt)^2ds\leq  \tilde{C}(1+\abs{x}^l)(\dt).
	\end{align}
	This is then followed by
	\begin{align*}
		&\lim_{n\to\infty}\frac{1}{n}\sum_{k=1}^{n}\abs{\frac{1}{\tau}\int_0^\tau\E\phi(\hx_{s}^{s-k\tau}(x))ds-\frac{1}{\tau}\int_0^\tau\mathcal{U}(s,s-k\tau,x)ds}\leq \tilde{C}(1+\abs{x}^l)(\dt).
	\end{align*}
	It then follows from (\ref{law of large numbers for exact solution}) and a triangle inequality argument that 
	\begin{align}\label{5.9a}
		\abs{\lim_{n\to\infty}\frac{1}{n}\sum_{k=1}^{n}\frac{1}{\tau}\int_0^\tau\phi\rbrac{\hx_{s+kN\dt}^{s,x}}ds-\int_\mathbb{S}\tilde{\phi}(\tilde{x})d\apm(\tilde{x})}=\mathcal{O}(\dt),   \quad a.s.,
	\end{align}
	Recall (\ref{law of large numbers for exact solution}) and (\ref{law of large numbers for numerical approximation}).
	Then by a triangle inequality argument, we obtain
	\begin{align}\label{20204g}
		\abs{\int_{\R^d}{\phi}({x})\bar{\rho}(d{x})-\int_{\R^d}{\phi}({x})\bar{\hat \rho}^\dt(d{x})}\leq 2\varepsilon+ \tilde{C}(1+\abs{x}^l)\dt.
	\end{align}
	Thus (\ref{error between average of periodic measure and its numerical approximation}) follows
	as the left hand side of 
	(\ref{20204g}) is independent $x$ and $\varepsilon$. \end{proof}
\begin{remark}
	Consider the modification of Benzi-Parisi-Sutera-Vulpiani's stochastic resonance model mentioned in the introduction, the coefficients of Linear growth is estimated as $L\leq \frac{\partial b(t,x)}{\partial_x}\big\vert_{x=4}\leq 48$. Hence the step size $\dt< \frac{1}{24}$ will satisfy our theorem.
\end{remark}

\section{Numerical examples}
In this section, we carry out some numerical experiments to support the theoretical results obtained in the last section. We give the error analysis for the  numerical scheme  of the average of periodic measures of two specific models arising in modelling daily temperature and climate dynamics respectively. For each example, we firstly generate discrete random periodic paths $\hx_{s+k\dt}^{s,x}$, $k=0,\ldots,N-1$ and test the convergence with different initial values. The numerical error we calculate in this section is 
\begin{align}\label{numerical error}
	\abs{\frac{1}{N}\sum_{k=0}^{N-1}\phi\rbrac{\hx_{s+k\dt}^{s,x}}-\int_\mathbb{S}{\phi}({x})d\bar{\rho}({x})},
\end{align}
which consists of three parts of errors: those influenced by the finiteness of $N$, the discretization error of time integral on $s$ in (\ref{5.9a}) and the error given in (\ref{5.9a}). Here $\phi\in C_p^{\infty}$. The main task is to estimate the error in (\ref{numerical error}) in terms of rate with respect to $\dt$. We choose large enough $N$ to reduce its impact on the error. We also compare numerically the errors with different initial time $s$ and find that the convergence of solutions to the random periodic paths in both models are very fast, as seen in Figure \ref{figure2} and Figure \ref{figure5}, so the effect of the initial time and position to the overall error is negligible as we take the average over a large number of iterations.

To carry out numerical experiments, we use Python 3.8.6 on Linux Fedora 32 with 3.40 GHz Intel(R) Core(TM) i7-3770 CPU and RAM 32.00 GB. There are two cores having higher computing speed (2958.762 MHz and 2121.630 MHz) compared with others (1600 MHz). We would not feel surprise to notice some abnormal computing times.
\begin{example}\label{example1}
	To present the error of our approximation scheme, we study the following temperature model considered by F. Benth and J. Benth \cite{Benth-Benth},
	\begin{align*}
		dX_t=(a_0+a_1 \cos(2\pi (t-a_2)/365) - \pi X_t)dt+\sigma dW_t,
	\end{align*}
	with $a_0 = 6.4$, $a_1=10.4$, $a_2=-166$ and $\sigma=0.3$. From the discussion in \cite{Feng-Zhao-Zhong2019measure}, it is known that the periodic measure of this model exists and is a Gaussian distribution with mean $$\frac{a_0}{\pi}+a_1 \frac{k \sin(k(t-a_2))+\pi \cos(k(t-a_2))}{k^2+\pi^2},$$
and  variance $\frac{\sigma^2}{2\pi}$, 
where $k=\frac{2\pi}{365}$. This is the case where the periodic measure is known explicitly. But a numerical experiment of calculating the numerical error is carried out here in order to verify the accuracy of our scheme. To simplify the calculation, we take the test function $\phi(x)=x^2$ as $\E \sbrac{X^2}=\E\sbrac{X}^2+\text{Var}\sbrac{X}$. Under the average of periodic measure, one can derive the exact value $$\int_{\R^d}x^2\bar{\rho}(d{x}) = \frac{a_0^2}{\pi^2}+\frac{a_1^2}{2(\pi^2+k^2)}+\frac{\sigma^2}{2\pi}.$$
On the other hand, conducting numerical approximation with Euler-Maruyama scheme, we obtain $\hx_{s+k\dt}^{s,x}$ for a range of different $\dt<2/\pi=0.63662$ varying from $0.2$ to $0.004$. We apply the Euler-Maruyama scheme with same length of time of 10000 periods for different step size $\dt$. The error is presented in Table \ref{table1} and as a log-log graph in Figure \ref{figure1}. Our numerical results show very good order 1 line fitting. Note the exact value $\int_{\R^d}x^2\bar{\rho}(d{x})=10.266021$ (rounded off in 6 decimal places). 

In Figure \ref{figure2}, we present two numerical approximations, $\hx_{s}^{0,-10}$ and $\hx_{s}^{0,10}$, to random periodic path with different initial values. These two trajectories are generated with the same realisation of noise. They merge before $t=2$ and show inconspicuous difference after that. We then generate one numerical approximation to random periodic path with 10000 periods and step size $\dt=0.01$. We collect the points on time $t=k\tau$ to build the histogram in left hand side graph of Figure \ref{figure3}(\subref{fig1:rho_0}) compared with its theoretical result $\rho_0$. In Figure \ref{figure3}, we give 8 more comparisons between numerical approximations to the periodic measure and its theoretical results. 
\end{example}
\begin{table}
	\begin{tabular}{ |c||r|r|r|r|r|  }
		\hline
		Step sizes      & $\dt=0.2$& $\dt=0.1$&$\dt=0.08$&$\dt=0.05$&$\dt=0.04$\\
		\hline
		Approximation& 10.560044& 10.383899& 10.357855& 10.318126& 10.307453\\
		Numerical error           &  0.294024&  0.117879&  0.091834&  0.052105&  0.041433\\
		CPU(seconds)    &    175.41&    371.39&    465.71&    756.28&    951.37\\
		\hline
		Step sizes      &$\dt=0.02$&$\dt=0.01$&$\dt=0.008$&$\dt=0.005$&$\dt=0.004$\\
		\hline
		Approximation& 10.283706& 10.277765& 10.275399& 10.271213& 10.270654\\
		Numerical error           &  0.017686&  0.011745&  0.009379&  0.005192&  0.004634\\
		CPU(seconds)    &   1600.47&   3711.75&   4407.75&   6551.42&   7600.27\\
		\hline
	\end{tabular}
	\caption{Numerical results of Example \ref{example1} where approximations and numerical errors are rounded off in 6 decimal places}	\label{table1}
\end{table}

\begin{figure}
	\centering
	\includegraphics[width=0.6\textwidth]{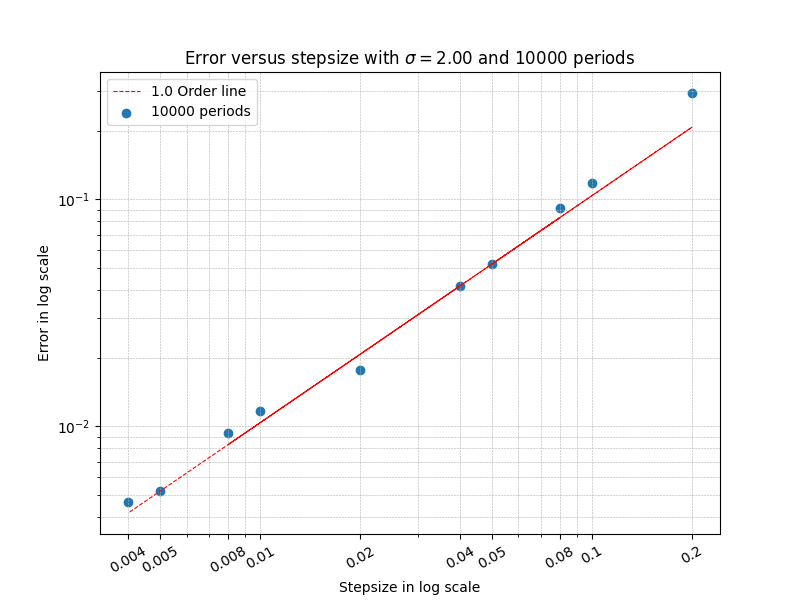}
	\caption{Error of approximation to the average of periodic measure versus step size in log-log graph (Example \ref{example1})}\label{figure1}
\end{figure}
\begin{figure}
	\centering
	\includegraphics[width=0.8\textwidth]{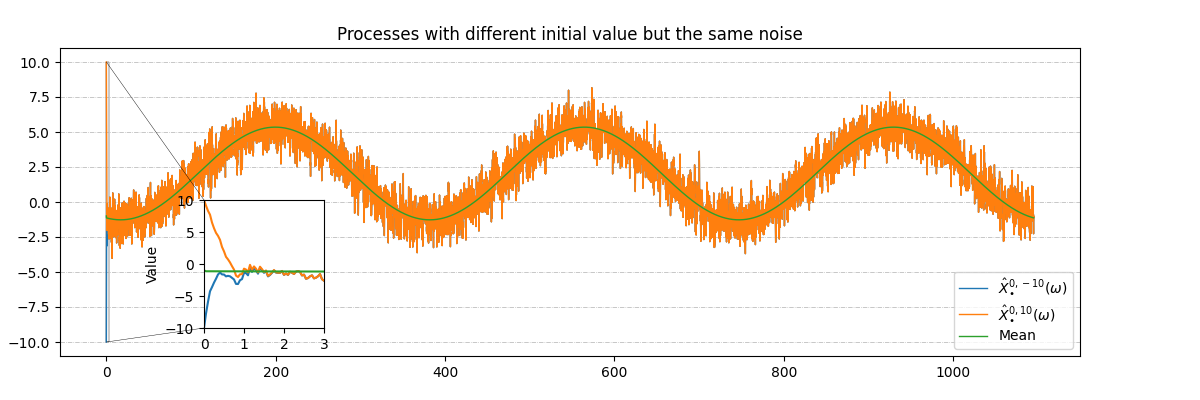}
	\caption{Paths of the temperature model (Example \ref{example1})}\label{figure2}
\end{figure}
\begin{figure}
	\centering
	\begin{subfigure}[b]{0.3\textwidth}
		\centering
		\includegraphics[width=\textwidth]{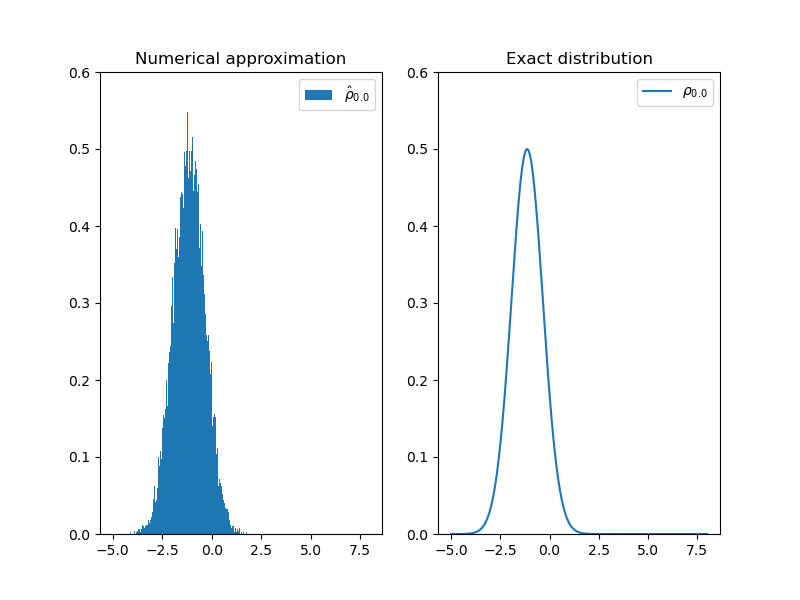}
		\caption{$\hat{\rho}_0$ and $\rho_0$}
		\label{fig1:rho_0}
	\end{subfigure}
	\hfill
	\begin{subfigure}[b]{0.3\textwidth}
		\centering
		\includegraphics[width=\textwidth]{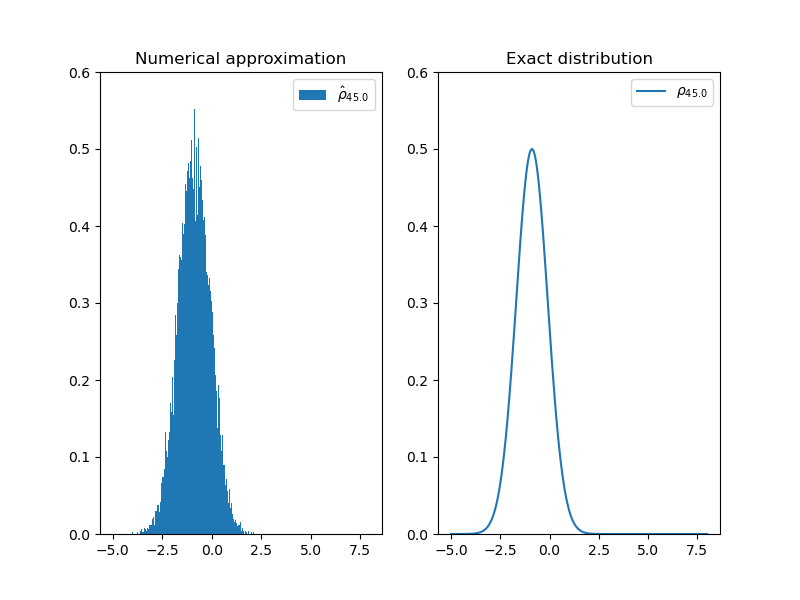}
		\caption{$\hat{\rho}_{45}$ and $\rho_{45}$}
		\label{fig1:rho_45}
	\end{subfigure}
	\hfill
	\begin{subfigure}[b]{0.3\textwidth}
		\centering
		\includegraphics[width=\textwidth]{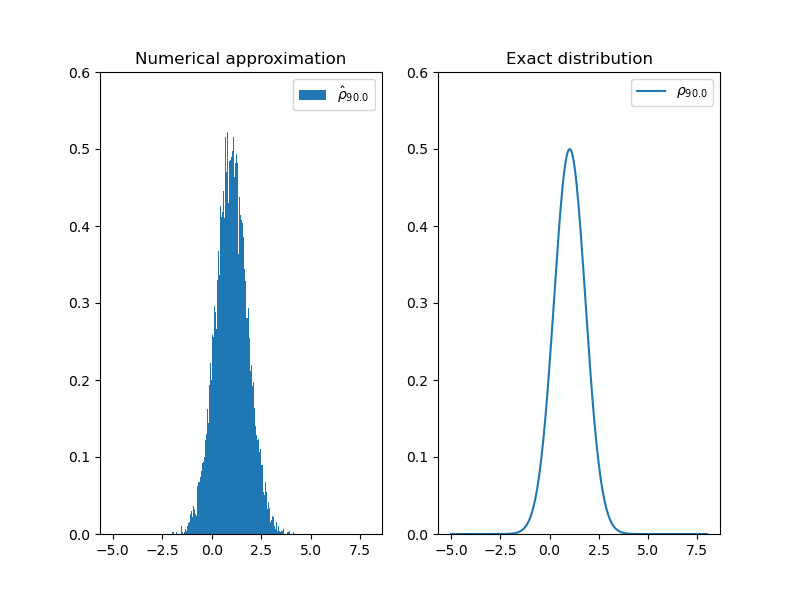}
		\caption{$\hat{\rho}_{90}$ and $\rho_{90}$}
		\label{fig1:rho_90}
	\end{subfigure}
	\centering
	\begin{subfigure}[b]{0.3\textwidth}
		\centering
		\includegraphics[width=\textwidth]{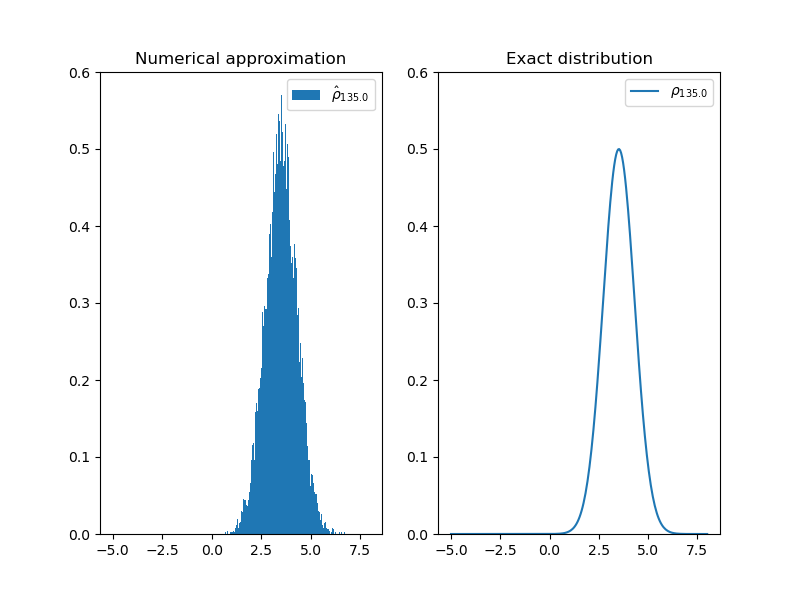}
		\caption{$\hat{\rho}_{135}$ and $\rho_{135}$}
		\label{fig1:rho_135}
	\end{subfigure}
	\hfill
	\begin{subfigure}[b]{0.3\textwidth}
		\centering
		\includegraphics[width=\textwidth]{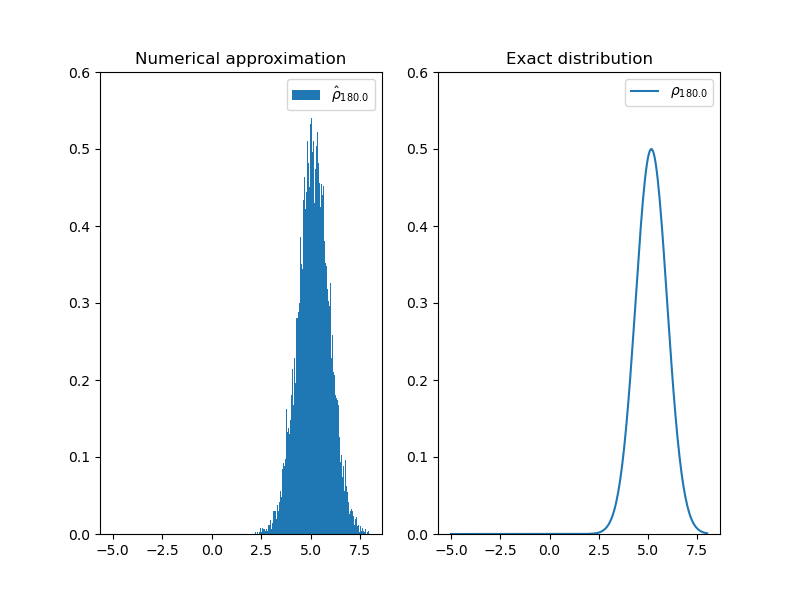}
		\caption{$\hat{\rho}_{180}$ and $\rho_{180}$}
		\label{fig1:rho_180}
	\end{subfigure}
	\hfill
	\begin{subfigure}[b]{0.3\textwidth}
		\centering
		\includegraphics[width=\textwidth]{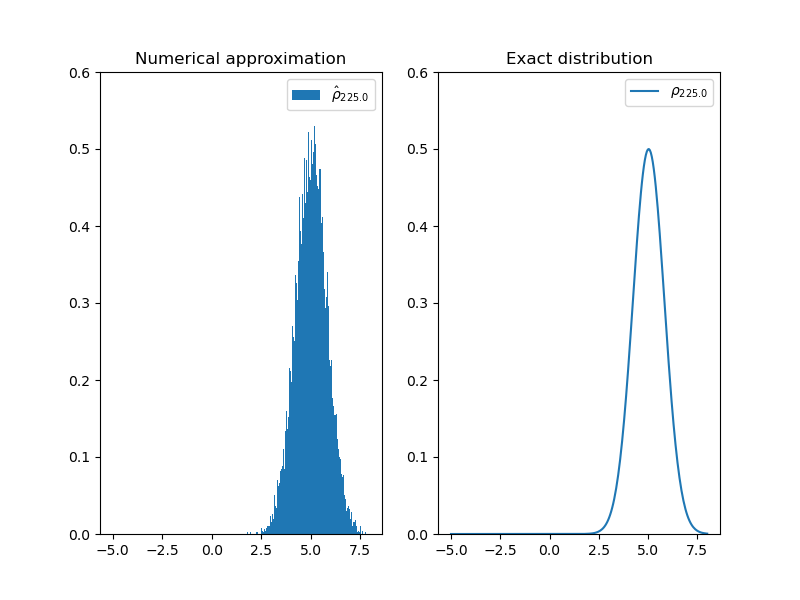}
		\caption{$\hat{\rho}_{225}$ and $\rho_{225}$}
		\label{fig1:rho_225}
	\end{subfigure}
	\centering
	\begin{subfigure}[b]{0.3\textwidth}
		\centering
		\includegraphics[width=\textwidth]{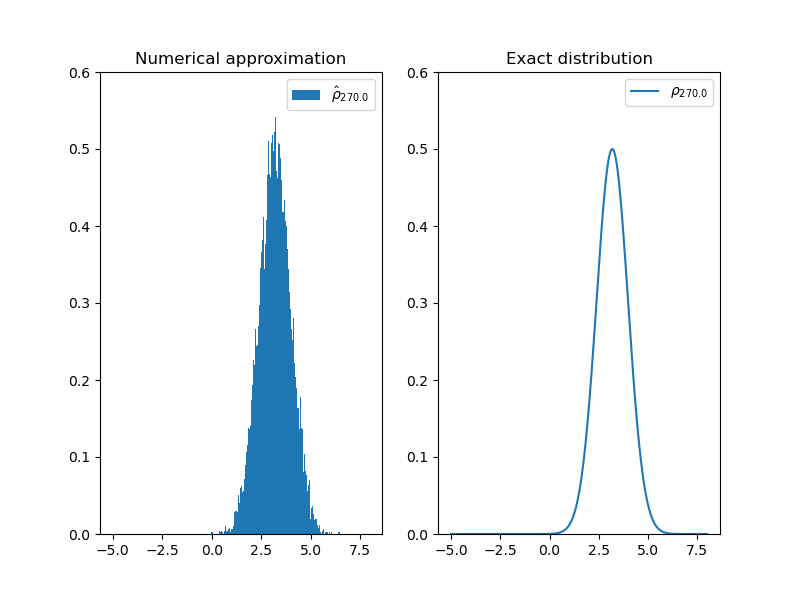}
		\caption{$\hat{\rho}_{270}$ and $\rho_{270}$}
		\label{fig1:rho_270}
	\end{subfigure}
	\hfill
	\begin{subfigure}[b]{0.3\textwidth}
		\centering
		\includegraphics[width=\textwidth]{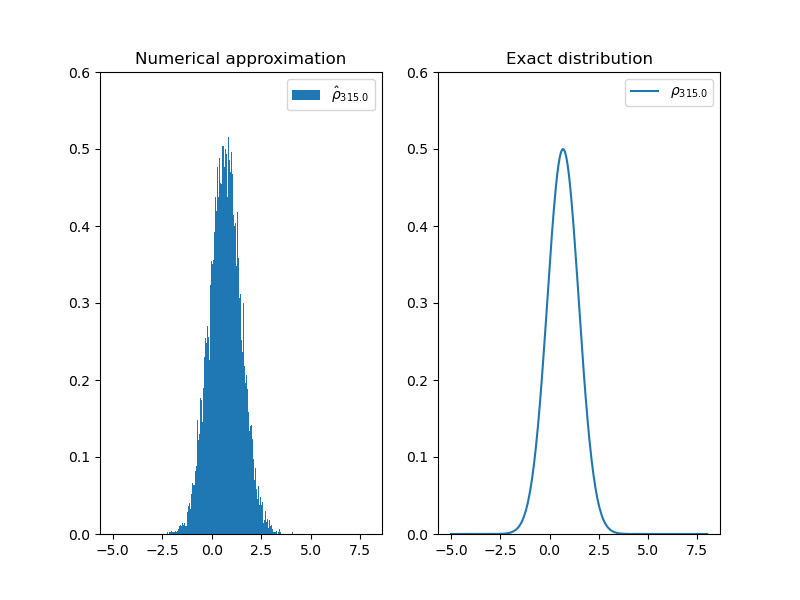}
		\caption{$\hat{\rho}_{315}$ and $\rho_{315}$}
		\label{fig1:rho_315}
	\end{subfigure}
	\hfill
	\begin{subfigure}[b]{0.3\textwidth}
		\centering
		\includegraphics[width=\textwidth]{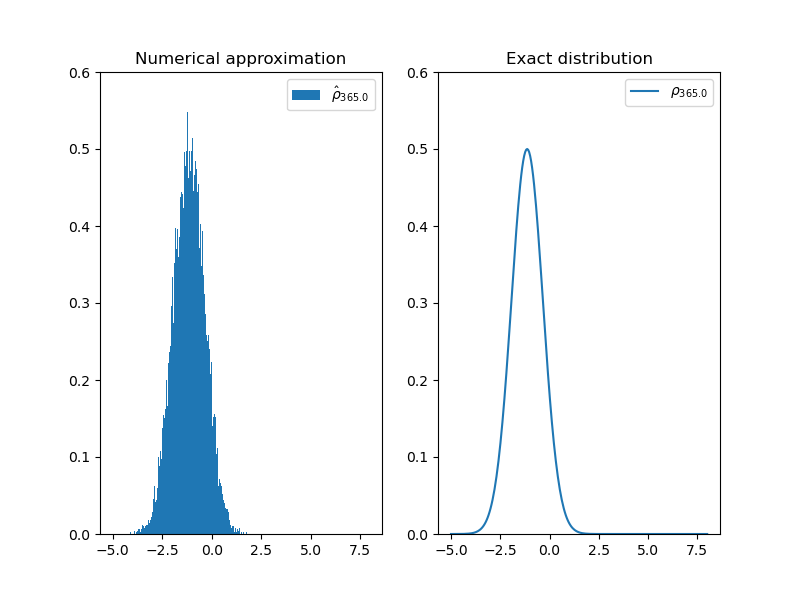}
		\caption{$\hat{\rho}_{365}$ and $\rho_{365}$}
		\label{fig1:rho_365}
	\end{subfigure}
	\caption{Comparisons between approximation of periodic measure with $\dt=0.01$ and 10000 periods and its theoretical result (Example \ref{example1})}
	\label{figure3}
\end{figure}

In practice, we run the computation with 10 different step sizes simultaneously under "multiprocessing" package of Python with 7 cores of CPU. The above results of error analysis took 7600.286 seconds of computing time where the CPU time of each step size is given in Table \ref{table1}. One can see the majority of time was consumed under the small step sizes such as $\dt=0.004$ and $\dt=0.005$. 

If necessary, one can split the approximation of random periodic path with small step sizes into several jobs. This works well due to ergodicity and fast convergence to random periodic path under our scheme. We do not need it in Example \ref{example1} as the computing time is reasonably short. 
But the possibility to split the computation into several independent jobs plays a crucial role in the case when a model has a large period. For such a problem, we need to consider the long time behaviour of $N\tau$ where both $N$ and $\tau$ are large. We will see that in the following example.

\begin{example}\label{example2}
We consider Benzi-Parisi-Sutera-Vulpiani's climate dynamics model given by SDE (\ref{equations of SDE}) with $b(t,x)=x-x^3+A\cos (B t)$ and $\sigma (x)=\sigma$. The coefficients are chosen as $A=0.12$, $B=0.001$ and $\sigma=0.285$ as discussed in \cite{Cherubini-Lamb-Rasmussen-Sato2017} and \cite{Feng-Zhao-Zhong2019resonance}. We make a time scaling by taking $b(t,x)=0.4\pi(x-x^3+0.12\cos(0.0004\pi t))$ and $\sigma(x)=0.285\times\sqrt{0.4\pi}\approx 0.3195$, so the period $\tau=5000$ in the system. Thus when we apply numerical approximation to the model, we can ensure our time step size dividing the period $\tau=5000$. It is also mollified to satisfy the global Lipschitz assumption in this paper. For this, what we could do is to modify the function $x-x^3$ by a linear function $128-47x$ when $x\geq 4$ and $-128-47x$ when $x\leq -4$ and smooth this function by mollifier $\eta_\varepsilon(x)=\eta(x)$, where 
$
	\eta(x)=
	\begin{cases}
		C\exp\left(\frac{1}{\abs{x}^2-1}\right)&\abs{x}\leq 1\\
		0&otherwise
	\end{cases} 
$ and $C$ is chosen such that $\int \eta(x)dx=1$. But this adds a lot of computing time as integration of convolution is needed in every step of the computation.

In our approximation, the drift term is $\hat b(t,x)=0.4\pi((1-\exp(-\frac{50}{x^2}))(x-x^3)+0.12\cos(0.0004\pi t))$. This function makes a very good approximation to function $b(t,x)$ when $\abs{x}\leq 4$ though it is not the case globally. Note this function is Lipschitz and smooth. As we mentioned in the introduction, our modified model provides the same climate dynamics as the original one of Benzi-Parisi-Sutera-Vulpiani. Our numerical simulations presented in Figure \ref{figure5} for the modified equation provide strong evidence that is the case as seen in the Figure \ref{figure5} that the trajectory rarely goes outside of $[-4,4]$. In fact, during the approximation with 10000 periods, we did not find any point in the whole data set running out of this interval.
	
Note this model does not have an explicit solution, so we cannot carry out the error analysis as we did in Example \ref{example1}. To overcome this difficulty we use the approximation of solution with $\dt=0.001$  to replace our exact solution in the error analysis. To make the computation more efficient, we split the approximation involving $\frac{5000\times10000}{0.001} = 5\times 10^{10}$ iterations to 8 individual jobs of $6.25\times 10 ^{9}$ iterations with independent Brownian motions. The results are shown in the left hand side table of Table \ref{table2}. We then conduct the numerical experiment for step size varying from $1/125$ to $1/50$. The error is in the right hand side table of Table \ref{table2} and the log-log graph is presented in Figure \ref{figure4}. We carry out numerical simulation with 10000 periods for each step size and our total cost is 155547.50 seconds with 7 cores. 

We consider numerical simulation with $\dt=0.01$ to show the stochastic resonance phenomenon in Figure \ref{figure5}. The simulations start from different initial condition but the same realisation of noise in each sub-graph, where one can see the convergence to random periodic path is also very fast. Together with ergodicity, it provides the possibility of splitting 10000 periods into several independent approximations for the step size $\dt=0.001$. Without the split, our total computing time would be about $5\times 10^5$ seconds as 6 cores of CPU are idle for long time.
	\begin{table}
		\begin{tabular}{ |c||c|c|  }
			\hline
			& Result& CPU(seconds)\\
			\hline
			1& 1.0120872& 71050.63 \\
			2& 1.0121229& 70247.79 \\
			3& 1.0120366& 70782.51 \\
			4& 1.0121788& 70421.53 \\
			5& 1.0118910& 70679.84 \\
			6& 1.0117187& 69755.99 \\
			7& 1.0115235& 67978.79 \\
			8& 1.0118564& 60069.49 \\
			\hline
			\multicolumn{3}{|l|}{\shortstack{Mean value of $\dt=0.001$\\ is 1.0119269}}\\
			\hline
		\end{tabular}
		\begin{tabular}{ |c||r|r|r|  }
			\hline
			Step sizes      &$\dt=1/50$&$\dt=1/64$&$\dt=1/80$\\
			\hline
			Approximation   & 1.0104580& 1.0106331& 1.0109203\\
			Numerical error & 0.0014689& 0.0012938& 0.0010066\\
			CPU(seconds)    &  28630.88&  36177.61&  45509.97\\
			\hline
			Step sizes      &$\dt=1/100$&$\dt=1/125$&\\
			\hline
			Approximation   & 1.0111015& 1.0112500&\\
			Numerical error & 0.0008254& 0.0006769&\\
			CPU(seconds)    &  56691.58&  49771.05&\\
			\hline
		\end{tabular}
		\caption{Numerical results of Example \ref{example2} where numerical results and errors are rounded off in 7 decimal places}\label{table2}
	\end{table}
	\begin{figure}
		\centering
		\includegraphics[width=0.6\textwidth]{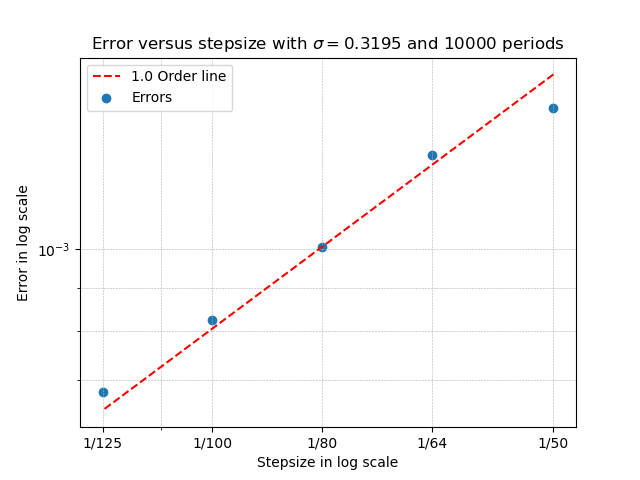}
		\caption{Error of approximation to average of periodic measure versus step size in log-log graph (Example \ref{example2})}\label{figure4} 
	\end{figure}

	 We also generate the periodic measure approximations from two paths each with 10000 periods and the step size $\dt=0.01$ under two different realisations of noise. The distributions of periodic measure are presented in Figure \ref{figure6}. One can see the distributions produced by two different Brownian motions are very similar. There are some minor differences due to insufficient amount of data. If we utilise sufficiently large amount of computations, the differences will eventually disappear. 
	\begin{figure}
		\centering
		\includegraphics[width=0.8\textwidth]{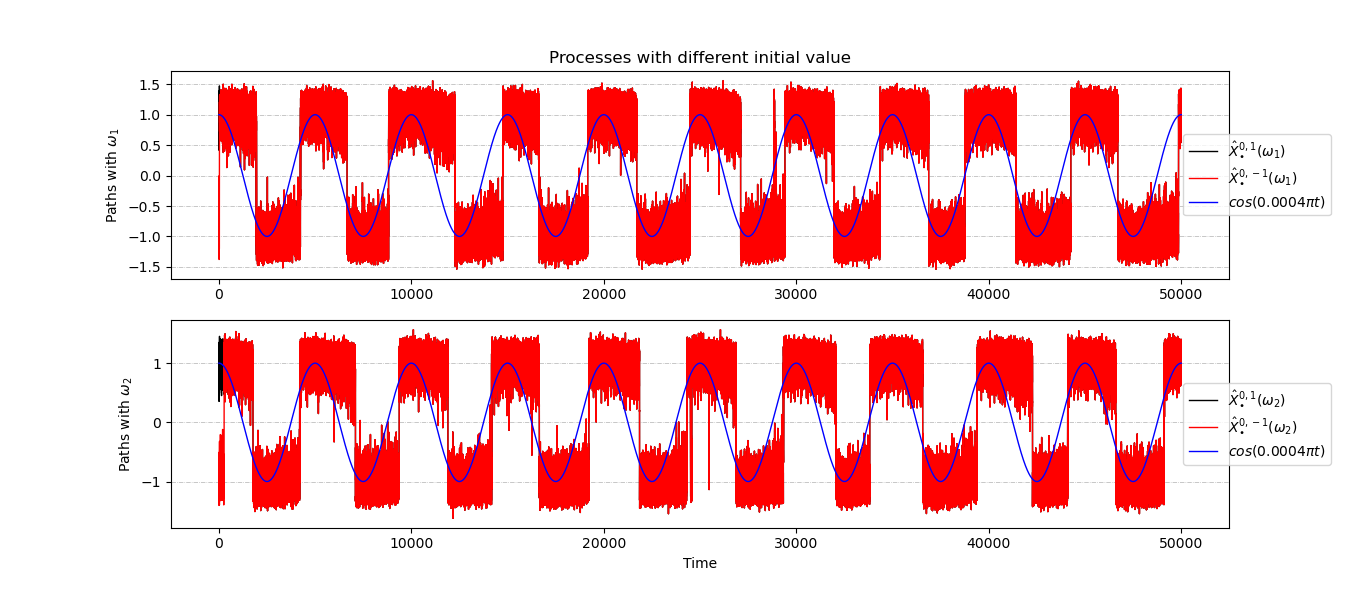}
		\caption{Paths of stochastic resonance model  (Example \ref{example2})}\label{figure5}
	\end{figure}
	
	We can apply time scaling on the model to rescale its period to a much smaller number. By doing this, the Lipschitz coefficient of the drift term will become very large. According to the upper bound of step size in Theorem \ref{theorem of main result}, the total cost of approximation will remain the same as the step size has to be very small. 


\end{example}

\section*{Acknowledgements} 
We would like to acknowledge the financial support of an EPSRC grant (ref. EP/S005293/2). We are very grateful to the referees for their constructive comments which led to significant improvements of this paper.

	\begin{figure}
		\centering
		\begin{subfigure}[b]{0.3\textwidth}
			\centering
			\includegraphics[width=\textwidth]{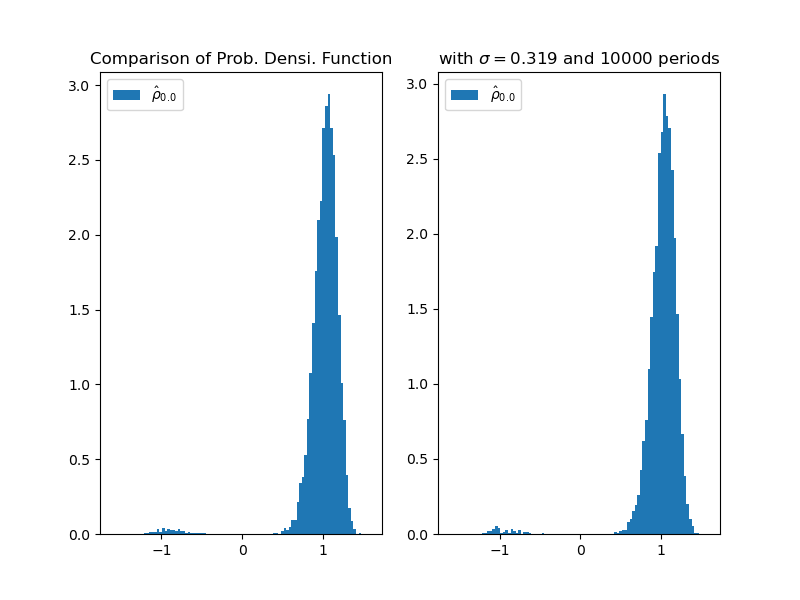}
			\caption{Distribution of $\hat\rho_0$}
			\label{fig:rho_0}
		\end{subfigure}
		\hfill
		\begin{subfigure}[b]{0.3\textwidth}
			\centering
			\includegraphics[width=\textwidth]{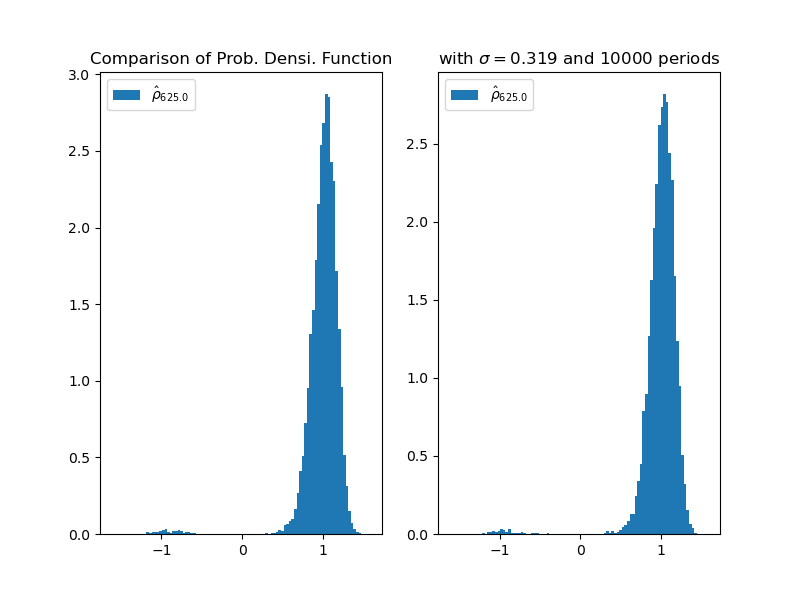}
			\caption{Distribution of $\hat\rho_{625}$}
			\label{fig:rho_625}
		\end{subfigure}
		\hfill
		\begin{subfigure}[b]{0.3\textwidth}
			\centering
			\includegraphics[width=\textwidth]{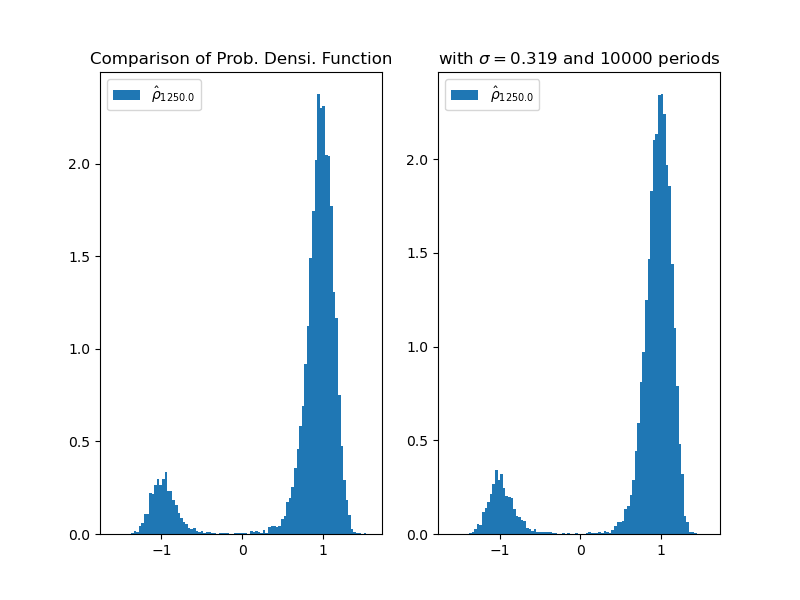}
			\caption{Distribution of $\hat\rho_{1250}$}
			\label{fig:rho_1250}
		\end{subfigure}
		\centering
		\begin{subfigure}[b]{0.3\textwidth}
			\centering
			\includegraphics[width=\textwidth]{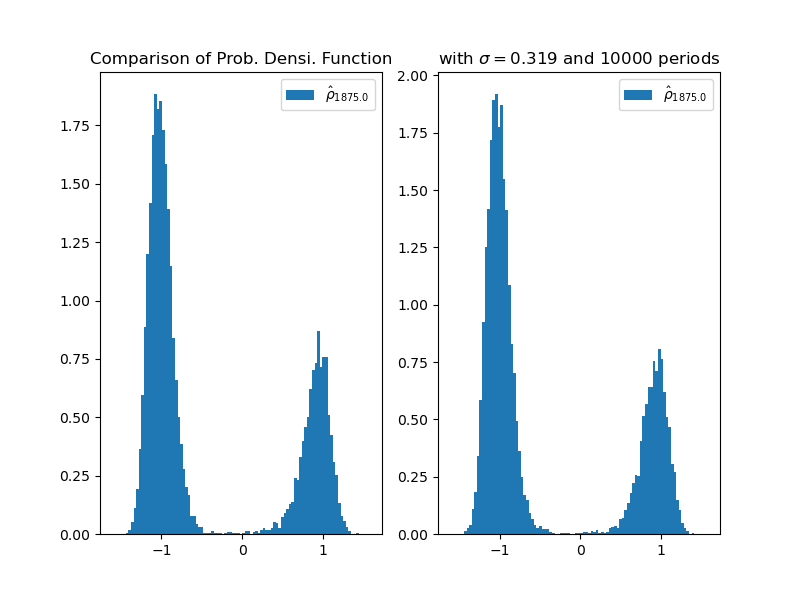}
			\caption{Distribution of $\hat\rho_{1875}$}
			\label{fig:rho_1875}
		\end{subfigure}
		\hfill
		\begin{subfigure}[b]{0.3\textwidth}
			\centering
			\includegraphics[width=\textwidth]{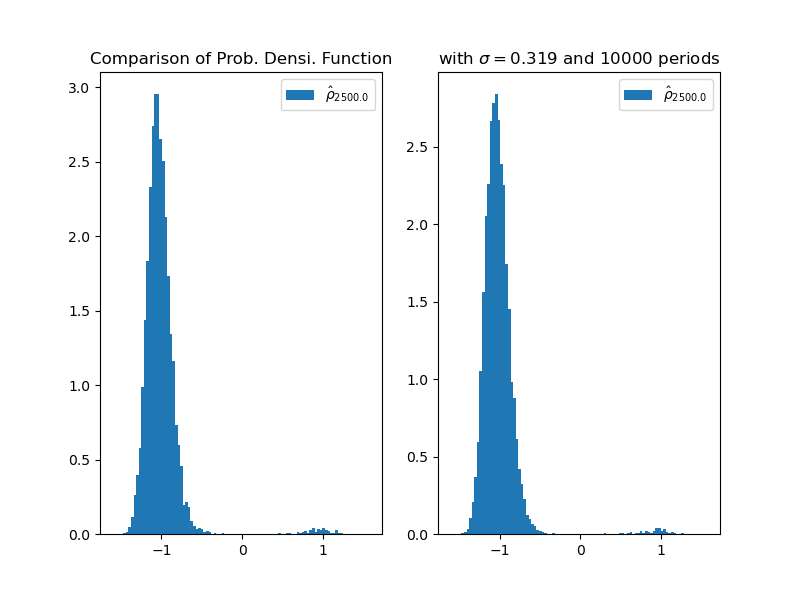}
			\caption{Distribution of $\hat\rho_{2500}$}
			\label{fig:rho_2500}
		\end{subfigure}
		\hfill
		\begin{subfigure}[b]{0.3\textwidth}
			\centering
			\includegraphics[width=\textwidth]{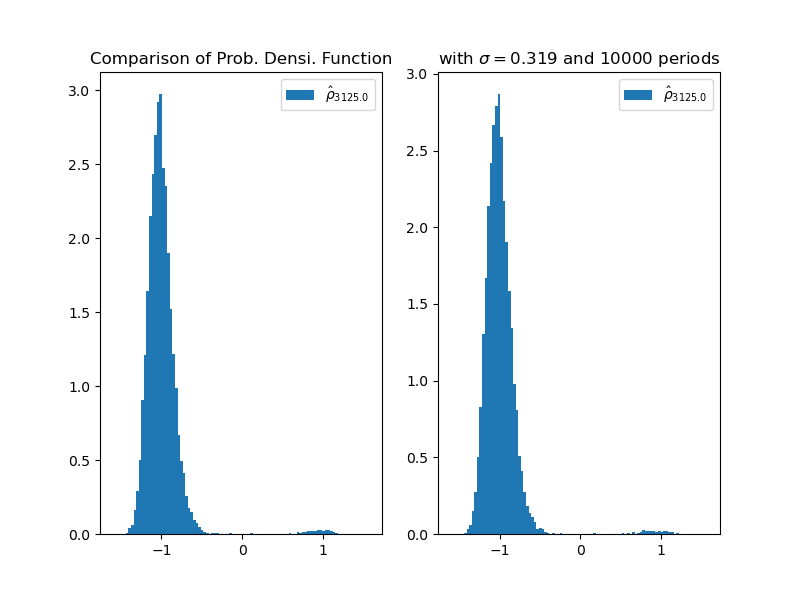}
			\caption{Distribution of $\hat\rho_{3125}$}
			\label{fig:rho_3125}
		\end{subfigure}
		\centering
		\begin{subfigure}[b]{0.3\textwidth}
			\centering
			\includegraphics[width=\textwidth]{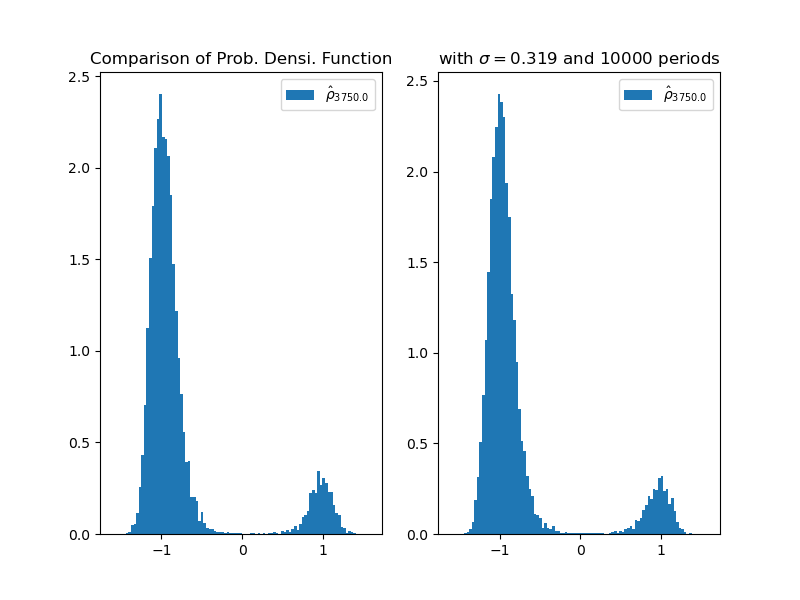}
			\caption{Distribution of $\hat\rho_{3750}$}
			\label{fig:rho_3750}
		\end{subfigure}
		\hfill
		\begin{subfigure}[b]{0.3\textwidth}
			\centering
			\includegraphics[width=\textwidth]{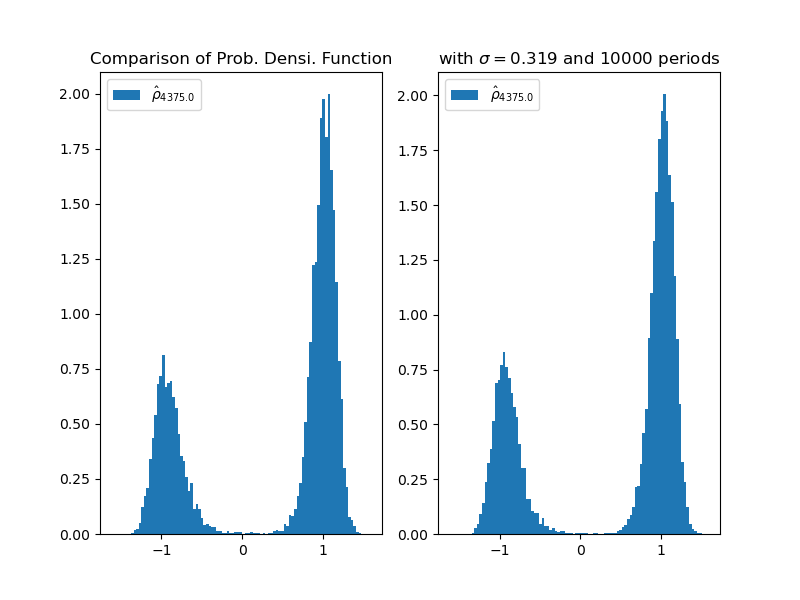}
			\caption{Distribution of $\hat\rho_{4375}$}
			\label{fig:rho_4375}
		\end{subfigure}
		\hfill
		\begin{subfigure}[b]{0.3\textwidth}
			\centering
			\includegraphics[width=\textwidth]{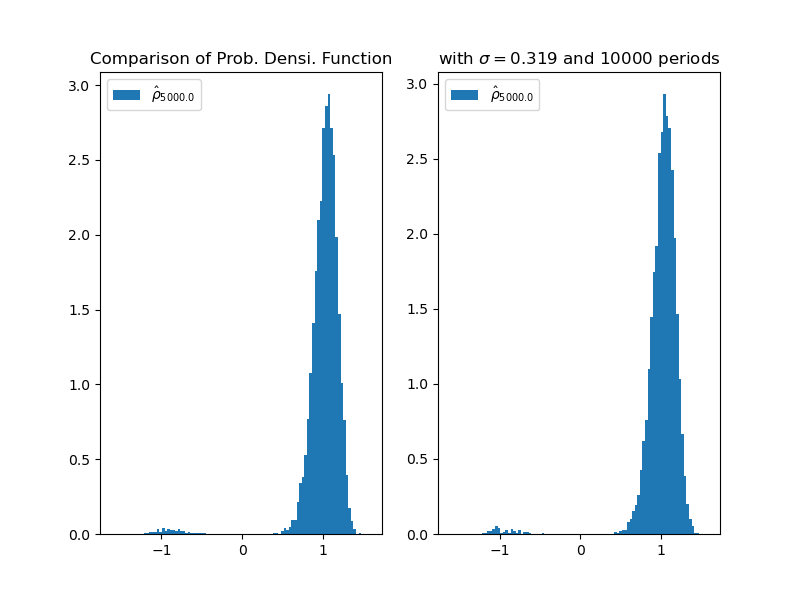}
			\caption{Distribution of $\hat\rho_{5000}$}
			\label{fig:rho_5000}
		\end{subfigure}
		\caption{Approximations of periodic measure with $\dt=0.01$ and 10000 periods (Example \ref{example2})}
		\label{figure6}
	\end{figure}
\end{document}